\documentclass{article}

\usepackage[centertags,leqno]{amsmath}
\usepackage{hyperref}
\usepackage{amsfonts}
\usepackage{amssymb}
\usepackage{amsthm}
\usepackage{newlfont}
\usepackage{amscd}

\usepackage[curve,arrow,cmactex,matrix]{xy}
\usepackage{verbatim}

\usepackage{color}

\usepackage{MnSymbol}
\usepackage{enumerate}

\usepackage{eucal}
\usepackage{mathrsfs}
\usepackage[utf8]{inputenc}

\newcommand{\QQ}{\mathbb{Q}}
\newcommand{\ZZ}{\mathbb{Z}}
\renewcommand{\AA}{\mathbb{A}}
\newcommand{\PP}{\mathbb{P}}

\newcommand{\FF}{\mathbb{F}}
\newcommand{\GG}{\mathbb{G}}

\newcommand{\OO}{\mathcal{O}}
\newcommand{\C}{\mathcal{C}}
\newcommand{\A}{\mathcal{A}}

\newcommand{\G}{\mathcal{G}}

\newcommand{\B}{\mathcal{B}}

\newcommand{\cA}{\mathcal{A}}
\newcommand{\cY}{\mathcal{Y}}
\newcommand{\cX}{\mathcal{X}}
\newcommand{\cZ}{\mathcal{Z}}
\newcommand{\cD}{\mathcal{D}}
\newcommand{\cO}{\mathcal{O}}
\newcommand{\cW}{\mathcal{W}}

\newcommand{\fm}{\mathfrak{m}}

\newcommand{\sA}{\mathscr{A}}
\newcommand{\sY}{\mathscr{Y}}
\newcommand{\sX}{\mathscr{X}}

\newcommand{\mN}{\mathrm{N}}

\newtheorem{thm}{Theorem}[section]
\newtheorem{cor}[thm]{Corollary}
\newtheorem{conj}[thm]{Conjecture}
\newtheorem{lem}[thm]{Lemma}
\newtheorem{prop}[thm]{Proposition}

\theoremstyle{definition}
\newtheorem{define}[thm]{Definition}
\newtheorem{const}[thm]{Construction}

\theoremstyle{remark}
\newtheorem{rem}[thm]{Remark}
\newtheorem{example}[thm]{Example}

\DeclareMathOperator{\Ker}{Ker}
\DeclareMathOperator{\Gal}{Gal}

\DeclareMathOperator{\Pic}{Pic}

\DeclareMathOperator{\Br}{Br}
\DeclareMathOperator{\loc}{loc}
\DeclareMathOperator{\spec}{spec}

\DeclareMathOperator{\divv}{div}

\DeclareMathOperator{\ev}{ev}

\DeclareMathOperator{\Hom}{Hom}

\DeclareMathOperator{\un}{un}

\DeclareMathOperator{\inv}{inv}

\DeclareMathOperator{\Iso}{Iso}

\DeclareMathOperator{\val}{val}
\DeclareMathOperator{\Sel}{Sel}
\DeclareMathOperator{\Frob}{Frob}

\DeclareMathOperator{\Kum}{Kum}
\DeclareMathOperator{\CT}{CT}
\DeclareMathOperator{\res}{res}
\DeclareMathOperator{\disc}{disc}
\DeclareMathOperator{\tor}{tor}

\DeclareMathOperator{\Aff}{Aff}

\DeclareMathOperator{\Mod}{mod}

\def\alp{{\alpha}}
\def\bet{{\beta}}
\def\gam{{\gamma}}
\def\del{{\delta}}
\def\eps{{\varepsilon}}

\def\lam{{\lambda}}
\def\sig{{\sigma}}
\def\vphi{{\varphi}}

\def\Om{{\Omega}}
\def\Gam{{\Gamma}}

\def\vphi{{\varphi}}
\def\div{\divv}

\def\lrar{\longrightarrow}

\def\hrar{\hookrightarrow}

\def\x{\stackrel}

\def\ovl{\overline}

\def\wtl{\widetilde}
\def\bksl{\;\backslash\;}

\DeclareFontEncoding{OT2}{}{} 
\DeclareTextFontCommand{\textcyr}{\fontencoding{OT2}\fontfamily{wncyr}\fontseries{m}\fontshape{n}\selectfont}

\newcommand{\Sha}{\textcyr{Sh}}
\AtBeginDocument{
\abovedisplayskip=5pt 
\abovedisplayshortskip=0pt 
\belowdisplayskip=5pt 
\belowdisplayshortskip=5pt 
}

\DeclareFontFamily{OT1}{mathc}{}
\DeclareFontShape{OT1}{mathc}{m}{n}{ <-> mathc10 }{}
\DeclareRobustCommand\xmcal[1]{\text{\usefont{OT1}{mathc}{m}{n}#1}}

\title{Second descent and rational points on Kummer varieties}

\author{Yonatan Harpaz }

\begin{document}
\maketitle

\begin{abstract}
A powerful method pioneered by Swinnerton-Dyer allows one to study rational points on pencils of curves of genus 1 by combining the fibration method with a sophisticated form of descent. A variant of this method, first used by Skorobogatov and Swinnerton-Dyer in 2005, can be applied to study rational points on Kummer varieties. In this paper we extend the method to include an additional step of second descent. Assuming finiteness of the relevant Tate-Shafarevich groups, we use the extended method to show that the Brauer-Manin obstruction is the only obstruction to the Hasse principle on Kummer varieties associated to abelian varieties with all rational 2-torsion, under mild additional hypotheses. 
\end{abstract}

\tableofcontents

\section{Introduction}\label{s:intro}

Let $k$ be a number field. A fundamental problem in Diophantine geometry is to determine for which geometric classes of smooth, proper and simply connected varieties over $k$, the Brauer-Manin obstruction is the only obstruction to the existence of a rational point. A geometric class which is expected to exhibit an extremely favorable behavior with respect to this question is the class of \textbf{rationally connected varieties}. Such varieties are always simply connected, and a conjecture of Colliot-Th\'el\`ene (\cite{CT01}) predicts that the set of rational points of a smooth, proper, rationally connected variety $X$ over $k$ is dense in the Brauer set of $X$. While this conjecture is still open, it has been established in many special cases. On the other extreme are simply-connected varieties of general type. For this class Lang's conjecture asserts that rational points are not Zariski dense, and their existence is not expected to be controlled at all by the Brauer-Manin obstruction (see \cite{SW95} and \cite{Sm14} for two kinds of conditional counter-examples). 

An intermediate class whose arithmetic is still quite mysterious is the class of simply connected Calabi-Yau varieties. In dimension $2$, these varieties are also known as \textbf{K3 surfacecs}. A conjecture far less documented than the two conjectures above predicts that the Brauer-Manin
obstruction is the only obstruction to the existence of rational points on K3 surfaces (see \cite[p. 77]{Sk09} and \cite[p. 484]{SZ}). The only evidence towards this conjecture is conditional, and relies on a method invented by Swinnerton-Dyer in~\cite{SD95}. In the realm of K3 surfaces there are two cases in which this method has been applied. The first case is when the K3 surface in question admits a fibration into curves of genus $1$ (see \cite{CTSSD98},\cite{SD00},\cite{CT01},\cite{Wi07}). In this case Swinnerton-Dyer's method depends on two big conjectures: the finiteness of Tate-Shafarevich groups of elliptic curves, and Schinzel's hypothesis. The second case is that of \textbf{Kummer surfaces} (\cite{SSD},\cite{HS15}). In this case the method does not require Schinzel's hypothesis (using, in effect, the only known case of the hypothesis, which is covered by Dirichlet's theorem), but only the Tate-Shafarevich conjecture.

Recall that a Kummer surface over $k$ is a K3 surface which is associated to a \textbf{$2$-covering} $Y$ of an abelian surface $A$, by which we mean a torsor under $A$ equipped with a map of torsors $p:Y \lrar A$ which covers the multiplication-by-$2$ map $A \lrar A$ (and so, in particular, $p$ is finite étale of degree $2^g$). The data of such a map is equivalent to the data to a lift of the class $[Y] \in H^1(k,A)$ to a class $\alp \in H^1(k,A[2])$. The antipodal involution $\iota_A=[-1]:A\lrar A$ then induces an involution $\iota_Y: Y \lrar Y$ and one defines the \textbf{Kummer surface} $X = \Kum(Y)$ as the minimal desingularisation of $Y/\iota_Y$. We note that the Kummer surface $X$ does not determine $A$ and $Y$ up to isomorphism over $k$. Indeed, for a quadratic extension $F/k$ one may consider the \textbf{quadratic twists} $A^F$ and $Y^F$ with respect to the $\ZZ/2$-actions given by $\iota_A$ and $\iota_Y$. We may then consider $Y^F$ as a torsor for $A^F$ determined by the same class $\alpha\in H^1(k,A^F[2])=H^1(k,A[2])$ and for every such $F/k$ we have a natural isomorphism $\Kum(Y^F) \cong \Kum(Y)$. The collection of quadratic twists $A^F$ can be organized into a fibration $\sA := (A \times \GG_m)/\mu_2 \lrar \GG_m/\mu_2 \cong \GG_m$, where $\mu_2$ acts diagonally by $(\iota_A,-1)$. In particular, for a point $t \in k^* = \GG_m(k)$, the fiber $\sA_t$ is naturally isomorphic to the quadratic twist $A^{k(\sqrt{t})}$. Similarly, we may organize the quadratic twists of $Y$ into a pencil $\sY \lrar \GG_m$ with $\sY_t \cong  Y^{F(\sqrt{t})}$. We may then consider the \textbf{entire family} $\sA_t$ as the family of abelian surfaces associated to $X$, and similarly the family $\sY_t$ as the family of $2$-coverings associated to $X$. 

When applying Swinnerton-Dyer method to a Kummer surface $X$, one typically assumes the finiteness of the $2$-primary part of the Tate-Shafarevich groups for all abelian surfaces associated to $X$ (in the above sense). We note that the finiteness of the $2$-primary part of $\sA_t$ is equivalent to the statement that the Brauer-Manin obstruction to the Hasse principle is the only one for any $2$-covering of $\sA_t$. In fact, to make the method work it is actually sufficient to assume that the Brauer-Manin obstruction is the only one for all the $\sY_t$ (as apposed to all $2$-coverings of all the $\sA_t$). Equivalently, one just needs to assume that $[\sY_t] \in H^1(k,\sA_t)$ is not a non-trivial divisible element of $\Sha(\sA_t)$ for any $t$. We may consequently consider a successful application of Swinnerton-Dyer's method to Kummer surfaces as establishing, \textbf{unconditionally}, cases of the following conjecture:
\begin{conj}\label{c:kummer}
Let $X$ be a Kummer surface over $k$. If the ($2$-primary part of) Brauer-Manin obstruction to the Hasse principle is the only one for all $2$-coverings associated to $X$, then the same holds for $X$.
\end{conj}

Conjecture~\ref{c:kummer} combined with the Tate-Shafarevich conjecture together imply that the Brauer-Manin obstruction controls the existence of rational points on Kummer surfaces. We may therefore consider any instance of Conjecture~\ref{c:kummer} as giving support for this latter claim, or more generally, support for the conjecture that the Brauer-Manin obstruction controls the existence of rational points on K3 surfaces. 

Let us now recall the strategy behind Swinnerton-Dyer's method. Let $Y$ be a $2$-covering of $A$ with associated class $\alpha\in H^1(k,A[2])$. To find a rational point on the Kummer variety $X=\Kum(Y)$, 
it is enough to find a rational point on a quadratic twist $Y^F$ for some $F/k$. At the first step of the proof,
using a fibration argument, one produces
a quadratic extension $F$ such that $Y^F$ is everywhere locally soluble.
Equivalently, $\alpha\in H^1(k,A^F[2])$ is in the \textbf{2-Selmer group} of $A^F$.
At the second step one modifies $F$ so that the 2-Selmer group of $A^F$
is spanned by $\alpha$ and the image of $A^F[2](k)$ under the Kummer 
map. This implies that $\Sha(A^F)[2]$ is spanned by the class $[Y^F]$, and is hence either $\ZZ/2$ or $0$. Now in all current applications of the method, as well as in the current paper, the abelian varieties under consideration are equipped with a principal polarization which is induced by a symmetric line bundle on $A$ (see Remark~\ref{r:alternating}). In that case it is known (see \cite{PS99}) that the induced Cassels--Tate pairing on $\Sha(A^F)$ is \textbf{alternating}. If one assumes in addition that the $2$-primary part of $\Sha(A^F)$ is finite then the $2$-part of the Cassels-Tate pairing is non-degenerate and hence the $2$-rank of $\Sha(A^F)[2]$ is even. The above alternative now implies that $\Sha(A^F)[2]$ is trivial and $[Y^F] = 0$, i.e., $Y^F$ has a rational point. Alternatively, instead of assuming that the $2$-primary part of $\Sha(A^F)$ is finite one may assume that $[Y^F]$ is not a non-trivial divisible element of $\Sha(A^F)$ (as is effectively assumed in Conjecture~\ref{c:kummer}). Indeed, the latter implies the former when $\Sha(A^F)[2]$ is generated by $[Y^F]$.
 
The process of controlling the $2$-Selmer group of $A^F$ while modifying $F$ can be considered as a type of \textbf{$2$-descent procedure} done ``in families''. In his paper~\cite{SD00}, Swinnerton-Dyer remarks that in some situations one may also take into account considerations of \textbf{second descent}. This idea is exploited in~\cite{SD00} to show a default of weak approximation on a particular family of quartic surfaces, but is not included systematically as an argument for the \textbf{existence} of rational points. As a main novelty of this paper, we introduce a form of Swinnerton-Dyer's method which includes a built-in step of ``second $2$-descent in families''. This involves a somewhat delicate analysis of the way the Cassels-Tate pairing changes in families of quadratic twists. It is this step that allows us to obtain Theorem~\ref{t:main-intro} below under reasonably simple assumptions, which resemble the type of assumptions used in~\cite{SSD}, and does not require an analogue of~\cite{SSD}'s Condition (E). Beyond this particular application, our motivation for introducing second descent into Swinnerton-Dyer's method is part of a long term goal to obtain a unified method which can be applied to an as general as possible Kummer surface. 
In principle, we expect the method as described in this paper and the method as appearing in~\cite{HS15} to admit a common generalization, which can be applied, say, to certain cases where the Galois module $A[2]$ is \textbf{semi-simple}, specializing in particular to the cases of~\cite{HS15} (where the action is simple) and the cases of~\cite{SSD} and the current paper (where the action is trivial).

With this motivation in mind, our goal in this paper is to prove Conjecture~\ref{c:kummer} for a certain class of Kummer surfaces. Let $f(x) = \prod_{i=0}^5 (x-a_i) \in k[x]$ be a polynomial of degree $6$ which splits completely in $k$ and such that $d := \prod_{i < j}(a_j - a_i) = \sqrt{\disc(f)} \neq 0$. Let $b_0,...,b_5 \in k^*$ be elements such that $\prod_i b_i \in (k^*)^2$ and consider the surface $X \subseteq \PP^5$ given by the smooth complete intersection
\begin{equation}\label{e:kummer}
\sum_i \frac{b_ix^2}{f'(a_i)} = \sum_i\frac{a_ib_ix^2}{f'(a_i)} = \sum_i\frac{a_i^2b_ix^2}{f'(a_i)} = 0 .
\end{equation}
The surface $X$ is an example of a K3 surface of degree $8$, and is also a Kummer surface whose associated family of abelian surfaces is the family of quadratic twists of the Jacobian $A$ of the hyperelliptic curve $y^2 = f(x)$. If we denote by $W$ the set of roots of $f$ then we may identify $A[2]$ with the submodule of $\mu_2^W/(-1,-1,...,-1)$ spanned by those vectors $(\eps_0,...,\eps_5) \in \mu_2^W$ such that $\prod_i\eps_i = 1$. The classes $[b_i] \in H^1(k,\mu_2)$ then naturally determine a class $\alp \in H^1(k,A[2])$, and the family of $2$-coverings associated to $X$ is exactly the family of $2$-coverings determined by this class. Our main result is then the following:
\begin{thm}\label{t:main-intro}
Assume that the classes of $\frac{b_1}{b_0},...,\frac{b_4}{b_0}$ are linearly independent in $k^*/(k^*)^2$ and that there exist finite odd places $w_1,...,w_5$ such that for every $i=1,...,5$ we have:
\begin{enumerate}[(1)]
\item
The elements $\{a_0,...,a_5\}$ are $w_i$-integral and $\val_{w_i}(a_i-a_0) = \val_{w_i}d = 1$.
\item 
The elements $\frac{b_1}{b_0},...,\frac{b_4}{b_0}$ are all units at $w_i$ but are not all squares at $w_i$. 
\end{enumerate}
Then Conjecture~\ref{c:kummer} holds for the Kummer suface $X$ given by~\eqref{e:kummer}. In particular, if the $2$-primary Tate-Shafarevich conjecture holds for every quadratic twist of $A$ then the $2$-primary Brauer-Manin obstruction is the only obstruction to the Hasse principle on $X$.
\end{thm}

\medskip

The first known case of Conjecture~\ref{c:kummer} was established in~\cite{SSD}, which is also the first application of Swinnerton-Dyer's method to Kummer surfaces. In that paper, Skorobogatov and Swinnerton-Dyer consider K3 surfaces which are smooth and proper models of the affine surface
\begin{equation}\label{e:ssd}
y^2 = g_0(x)g_1(z)
\end{equation}
where $g_0,g_1$ are separable polynomials of degree $4$. These are in fact Kummer surfaces whose associated family of $2$-coverings is the family of quadratic twists of the surface $D_0 \times D_1$, where $D_i$ is the genus $1$ curve given by
$$ D_i: y^2 = g_i(x) .$$
The associated family of abelian surfaces is the family of quadratic twists of $E_0 \times E_1$, where $E_i$ is the Jacobian of $D_i$ and is given by 
$E_i: y^2 = f_i(x)$,
where $f_i$ is the cubic resolvant of $f_i$. Three types of conditions are required in~\cite{SSD}:
\begin{enumerate}[(1)]
\item
The curves $E_0,E_1$ have all their $2$-torsion defined over $k$, i.e., $f_0$ and $f_1$ split completely in $k$. Equivalently, the discriminants of $g_0$ and $g_1$ are squares and their splitting fields are at most biquadratic.
\item
Condition (Z). This condition asserts the existence, for each $i = 0,1$, of suitable multiplicative places $v_i,w_i$ for $E_i$, at which, in particular, $E_{1-i}$ has good reduction and the classes $\alp_0,\alp_1$ are non-ramified. It also implies that the $2$-primary part of the Brauer group of $X$ is algebraic.
\item
Condition (E). This condition, which is to some extent analogous to Condition (D) in applications of the method to pencil of genus 1 curves, implies, in particular, that there is no algebraic Brauer-Manin obstruction to the existence of rational points on $X$.
\end{enumerate}

Given a Kummer surface $X$ of the form~\eqref{e:ssd} satisfying the above conditions, the main result of~\cite{SSD} asset that Conjecture~\ref{c:kummer} holds for $X$. Even more, under conditions (1)-(3) above there is no $2$-primary Brauer-Manin obstruction on $X$. It then follows, and this is how the main theorem of~\cite{SSD} is actually stated, that under the Tate-Shafarevich conjecture for the quadratic twists of $E_0,E_1$, the \textbf{Hasse principle} holds for $X$.

The second case of conjecture~\ref{c:kummer} established in the literature appears in~\cite{HS15}, where the authors consider also \textbf{Kummer varieties}, i.e., varieties obtained by applying the Kummer construction to abelian varieties of arbitrary dimension. When restricted to surfaces, the results of~\cite{HS15} cover two cases:
\begin{enumerate}[(1)]
\item
The case where $X$ is of the form~\eqref{e:ssd} where now $g_0,g_1$ are polynomials whose Galois group is $S_4$. The only other assumption required in~\cite{HS15} is that there exist odd places $w_0,w_1$ such that $g_0$ and $g_1$ are $w_i$-integral and such that $\val_{w_i}(\disc(g_j)) = \del_{i,j}$ for $i,j=0,1$.
\item
The case where $X = \Kum(Y)$ and $Y$ is a $2$-covering of the Jacobian $A$ of a hyperelliptic curve $y^2 = f(x)$, with $f$ is an irreducible polynomial of degree $5$. In this case $X$ can be realized as an explicit complete intersection of three quadrics in $\PP^5$. It is then required there exists an odd place $w$ such that $f$ is $w$-integral and $\val_w(\disc(f)) = 1$, and such that the class $\alp \in H^1(k,A[2])$ associated to $Y$ is unramified at $w$. 
\end{enumerate}

\begin{rem}
While the main theorem of~\cite{HS15} can be considered as establishing Conjecture~\ref{c:kummer} for the Kummer surfaces of type (1) and (2), what it actually states is that under the $2$-primary Tate-Shafarevich conjecture (for the relevant abelian varieties) the Kummer surfaces of type (1) and (2) satisfy the Hasse principle. The gap between these two claims can be explained by a recent paper of Skorobogatov and Zarhin~\cite{SZ16}, which shows, in particular, that there is no $2$-primary Brauer-Manin obstruction for Kummer surfaces of type (1) and (2).
\end{rem}

\section{Main results}

While our main motivation in this paper comes from Kummer surfaces, it is often natural to work in the more general context of \textbf{Kummer varieties}. These are the higher dimensional analogues of Kummer surfaces which are obtained by applying the same construction to a $2$-covering $Y$ of an abelian variety $A$ of dimension $g \geq 2$. A detailed discussion of such varieties occupies the majority of~\S\ref{s:kummer}. For now, we will focus on formulating the main theorem of this paper in the setting of Kummer varieties and show how Theorem~\ref{t:main-intro} is implied by it. We begin with some terminology which will be used throughout this paper.

Let $k$ be a number field and let $A$ be a principally polarized abelian variety of dimension $g$ over $k$. Assume that $A[2](k) \cong (\ZZ/2)^g$, i.e., that $A$ has all of its $2$-torsion points defined over $k$. Let $\A$ be a Neron model for $A$. We will denote by $C_v$ the component group of the special fiber of $\A$ over $v$. Generalizing the ideas of~\cite{SSD}, we will need to equip $A$ with a collection of ``special places''. We suggest the following terminology:
\begin{define}\label{d:2-struct}
Let $A$ be an abelian variety whose $2$-torsion points are all rational. A \textbf{$2$-structure} on $A$ is a set $M \subseteq \Om_k$ of size $2g$ containing odd multiplicative places and such and the natural map
\begin{equation}\label{e:red-map}
A[2] \lrar \prod_{w \in M} C_w/2C_w
\end{equation}
is an \textbf{isomorphism}.
\end{define}

\begin{rem}\label{r:odd}
Since $A$ has all its $2$-torsion defined over $k$ the $2$-part of $C_w$ is non-trivial for any multiplicative odd place $w$. It follows that if $M \subseteq \Om_k$ is a $2$-structure then for each $w \in M$ the $2$-part of $C_w$ must be cyclic of order $2$. 
\end{rem}

To formulate our main result we will also need the following extension of the notion of a $2$-structure:
\begin{define}\label{d:2-struct-ext}
Let $A$ be an abelian variety as above. An \textbf{extended $2$-structure} on $A$ is a set $M \subseteq \Om_k$ of size $2g+1$ containing odd multiplicative places such that the natural map
\begin{equation}\label{e:red-map-ext}
A[2] \lrar \prod_{v \in M} C_w/2C_w
\end{equation}
is injective and its image consists of those vectors $(c_w)_{w \in M} \in \prod_{w \in M} C_w/2C_w$ for which $c_w \neq 0$ at an even number of $w \in M$. 
\end{define}

\begin{example}\label{e:curve}
Let $E$ be an elliptic curve given by $y^2 = (x-c_1)(x-c_2)(x-c_3)$. If $w_1,w_2,w_3$ are three places such that $\val_{w_i}(c_j-c_k) = 1$ for any permutation $i,j,k$ of $1,2,3$, and such that $\val_{w_i}(c_i-c_j) = 0$ for any two $i \neq j$, then $\{w_1,w_2,w_3\}$ constitutes an extended $2$-structure for $E$.
\end{example}

\begin{rem}\label{r:ext}
If $M \subseteq \Om_k$ is an extended $2$-structure for $A$ then for every $w \in M$ the subset $M \setminus \{w\}$ is a $2$-structure for $A$. Conversely, if a set $M$ of $2g+1$ places satisfies the property that $M \setminus \{w\}$ is a $2$-structure for every $w \in M$ then $M$ is an extended $2$-structure. 
\end{rem}

\begin{rem}\label{r:simple}
If $A$ carries an extended $2$-structure $M$ then $A$ is necessarily \textbf{simple} (over $k$). Indeed, if $A = A_1 \times A_2$ then for every place $w$ we have $C_w = C_{w,1} \times C_{w,2}$, where $C_{w,1},C_{w,2}$ are the corresponding groups of components for $A_1,A_2$ respectively. Since the $2$-part of $C_w$ for $w \in M$ is cyclic of order $2$ (Remark~\ref{r:odd}) we see that the $2$-part of $C_{w,i}$ must be cyclic of order $2$ for one $i \in \{1,2\}$ and trivial for the other. We may then divide $M$ into two disjoint subsets $M = M_1 \cup M_2$ such that for $w \in M_i$ we have $C_{w,j}/2C_{w,j} \cong \ZZ/2^{\del_{i,j}}$. By Remark~\ref{r:ext} $M \setminus \{w\}$ is a $2$-structure for every $w \in M$. It then follows that $|M_i \setminus \{w\}| = 2\dim(A_i)$ for every $i=1,2$ and every $w \in M$, which is of course not possible.
\end{rem}

\begin{define}[{cf. \cite[Definition 3.4]{HS15}}]\label{d:nondeg}
Let $M$ be a semi-simple Galois module and let $R$ be the endomorphism algebra of $M$ (in which case $R$ naturally acts on $H^1(k,M)$). We will say that $\alp \in H^1(k,M)$ is \textbf{non-degenerate} if the $R$-submodule generated by $\alp$ in $H^1(k,M)$ is free.
\end{define}

Definition~\ref{d:nondeg} will be applied to the Galois module $M = A[2]$, which in our case is a constant Galois module isomorphic to $\ZZ/2^n$, and so $R$ is the $n\times n$ matrix ring over $\ZZ/2$. In particular, if $\alp = (\alp_1,...,\alp_n) \in H^1(k,\ZZ/2^n) \cong H^1(k,\ZZ/2)^n$ is an element then $\alp$ is non-degenerate if and only if the classes $\alp_1,...,\alp_n \in H^1(k,\ZZ/2)$ are linearly independent. 

We are now ready to state our main result.
\begin{thm}\label{t:main}
Let $k$ be a number field and let $A_1,...,A_n$ be principally polarized abelian varieties over $k$ such that each $A_i$ has all its $2$-torsion defined over $k$. For each $i$, let $M_i \subseteq \Om_k$ be an extended $2$-structure on $A_i$ such that $A_j$ has good reduction over $M_i$ whenever $j \neq i$. Let $A = \prod_i A_i$ and let $\alp \in H^1(k,A[2])$ be a non-degenerate element which is unramified over $M = \cup_i M_i$ but has non-trivial image in $H^1(k_w,A[2])$ for each $w \in M$. Let $X_{\alp} = \Kum(Y_{\alp})$ where $Y_{\alp}$ is the $2$-covering of $A$ determined by $\alp$. Then Conjecture~\ref{c:kummer} holds for $X_{\alp}$. In particular, under the Tate-Shafarevich conjecture the $2$-primary Brauer-Manin obstruction is the only one for the Hasse principle on $X_{\alp}$.
\end{thm}

\begin{rem}
The proof of Theorem~\ref{t:main} actually yields a slightly stronger result: under the Tate-Shafarevich conjecture the $2$-primary \textbf{algebraic} Brauer-Manin obstruction is the only one for the Hasse principle on $X$ (see Remark~\ref{r:2-pri-alg}). In fact, one can isolate an explicit finite subgroup $\C(X_\alp) \subseteq \Br(X_\alp)$ (see Definition~\ref{d:C-alpha}) whose associated obstruction is, in this case, the only one for the Hasse principle.
\end{rem}

\begin{rem}
When $A$ is a product of two elliptic curves with rational $2$-torsion points one obtains the same type of Kummer surfaces as the ones studied in~\cite{SSD}. However, the conditions required in Theorem~\ref{t:main} are not directly comparable to those of~\cite{SSD}. On the one hand, Theorem~\ref{t:main} does not require any analogue of Condition (E). On the other hand, Theorem~\ref{t:main} requires each elliptic curve to come equipped with an extended $2$-structure (consisting, therefore, of three special places for each curve, see Example~\ref{e:curve}), while the main theorem of~\cite{SSD} only requires each elliptic curve to have a $2$-structure (consisting, therefore, of two special places for each curve). Modifying the argument slightly, one can actually make Theorem~\ref{t:main} work with only a $2$-structure for each $A_i$, at the expense of assuming some variant of Condition (E). In the case of a product of elliptic curves, this variant is slightly weaker than the Condition (E) which appears in~\cite{SSD}. This can be attributed to the existence of a phase of second descent, which does not appear in~\cite{SSD}. 
\end{rem}

We finish this section by showing how Theorem~\ref{t:main-intro} can be deduced from Theorem~\ref{t:main}. Let $f(x) = \prod_{i=0}^5 (x-a_i) \in k[x]$ be a polynomial of degree $6$ which splits completely in $k$ and such that $d := \prod_{i < j}(a_j - a_i) = \sqrt{\disc(f)} \neq 0$. Let $C$ be the hyperelliptic curve given by $y^2 = f(x)$ and let $A$ be the Jacobian of $C$. If we denote by $W = \{a_0,...,a_5\}$ the set of roots of $f$ then we may identify $A[2]$ with the submodule of $\mu_2^W/(-1,-1,...,-1)$ spanned by those vectors $(\eps_0,...,\eps_5) \in \mu_2^W$ such that $\prod_i\eps_i = 1$. Consequently, if we denote by $\G = k^*/(k^*)^2 = H^1(k,\mu_2)$ then we may identify 
$$ H^1(k,A[2]) \cong \left\{(\bet_0,...,\bet_5) \in \G^W \right|\left. \prod \bet_i = 1\right\} / \G .$$ 
In particular, we may represent element of $H^1(k,\mu_2)$ be vectors $\ovl{b} = (b_0,...,b_5) \in (k^*)^W$ such that $\prod_i b_i \in (k^*)^2$, defined up to squares and up to multiplication by a constant $b \in k^*$. In fact, we may always choose a representative such that $\prod_i b_i = 1$. We note that the corresponding element $\bet = [\ovl{b}] \in H^1(k,A[2])$ is non-degenerate (in the sense of Definition~\ref{d:nondeg} if and only if the classes $[b_1/b_0],...,[b_4/b_0]$ are linearly independent in $k^*/(k^*)^2$.

\begin{proof}[Proof of Theorem~\ref{t:main-intro} assuming Theorem~\ref{t:main}]
Let $Y$ be the $2$-covering of $A$ determined by the class $\bet = [\ovl{b}] \in H^1(k,A[2])$. Then the Kummer surface $X = \Kum(Y)$ is isomorphic to the smooth complete intersection~\eqref{e:kummer} by~\cite[Theorem 3.1]{Sk10}. 
The assumptions of Theorem~\ref{t:main-intro} now imply that $\bet$ is unramified over $M = \{w_1,...,w_5\}$ but is non-trivial at each $H^1(k_{w_i},A[w])$, and so the assumptions of Theorem~\ref{t:main} will hold as soon as we show that $M$ forms an extended $2$-structure for $A$. Now for each $i$ the polynomial $f$ is $w_i$-integral and the place $w_i$ satisfies $\val_{w_i}(a_i-a_0) = 1$ and $\val_{w_i}(a_j-a_{j'}) = 0$ whenever $j \neq j'$ and $\{j,j'\} \neq \{0,i\}$. Since $f$ is $w_i$-integral it determines an $w_i$-integral model $\C$ for $C$, and a local analysis shows that the reduction of $\C$ mod $w$ is a curve of geometric genus $1$ and a unique singular point $P$, which is also a rational singular point of the model $\C$. Blowing up at $P$ one obtains a regular model for $C$ at $w$ whose special fiber has two components (of genus $1$ and $0$ respectively) which intersect at two points. Using~\cite[Theorem 9.6.1]{BLR} we may compute that the group of components $C_{w_i}$ of a N\'eron model for $A$ is isomorphic to $\ZZ/2$. Now for each $i=1,...,4$ let $P_i \in A$ be the point corresponding to the formal sum $(a_i,0) - (a_5,0)$ of points of $C$. Then for $i,j = 1,...,4$ we have that the image of $P_i$ in $C_{w_j}$ is nontrivial if and only if $i=j$. On the other hand, all the four points $P_1,...,P_4$ map to the non-zero component of $w_5$. It then follows that the map 
$$ A[2] \lrar \prod_{i=1}^{5} C_{w_i} = \prod_{i=1}^{5} C_{w_i}/2C_{w_i} $$
is injective and its image consists of exactly those vectors $(c_1,...,c_5) \in \prod_{i=1}^{5} C_{w_i}$ in which an even number of the $c_v$'s are non-trivial.
\end{proof}

\section{Preliminaries}\label{s:prelim}

In this section we will established some preliminary machinery that will be used in \S\ref{s:proof} to prove Theorem~\ref{t:main}. We will begin in \S\ref{s:weil} by recalling the \textbf{Weil pairing} and establish some useful lemmas in the situation where all the $2$-torsion points of $A$ are defined over $k$. In \S\ref{s:CT} we simply recall a definition of the Cassels-Tate pairing via evaluation of Brauer elements. In \S\ref{s:kummer} we give a short introduction to Kummer varieties and consider cases where the Brauer elements appearing in \S\ref{s:CT} descent to the corresponding Kummer varieties. Finally, in \S\ref{s:MR} we recall the approach of Mazur and Rubin to the analysis of the change of Selmer groups in families of quadratic twists. While mostly relying on ideas from~\cite{MR10}, this section is essentially self contained, and we give detailed proofs of all the results we need. We then complement the discussion of Selmer groups in families of quadratic twist by considering the change of Cassels-Tate pairing under quadratic twist, using results of \S\ref{s:kummer}.

\subsection{The Weil pairing}\label{s:weil}

Let $A$ be an abelian variety over a number field $k$ and let $\hat{A}$ be its dual abelian variety. Recall the \textbf{Weil pairing}
$$ \left<,\right>: A[2] \times \hat{A}[2] \lrar \mu_2 ,$$
which is a perfect pairing of finite Galois modules. If $A$ is equipped with a principal polarization, i.e., a self dual isomorphism $\lam: A \x{\cong}{\lrar} \hat{A}$, then we obtain an isomorphism $A[2] \cong \hat{A}[2]$ and an induced alternating self pairing
\begin{equation}\label{e:weil}
\left<,\right>_{\lam}: A[2] \times A[2] \lrar \mu_2
\end{equation}
which is known to be alternating. We note that a principal polarization $\lambda$ induces a principal polarization $A^F \x{\cong}{\lrar} \hat{A}^F$ after quadratic twist by any quadratic extension $F/k$. To keep the notation simple we will use the same letter $\lam$ to denote all these principal polarizations. 

\begin{rem}\label{r:abuse}
The $2$-torsion modules $A[2]$ and $A^F[2]$ are \textbf{canonically} isomorphic for any $F/k$. We will consequently often abuse notation and denote by $A[2]$ the $2$-torsion module of any given quadratic twist of $A$.
\end{rem}

\begin{rem}\label{r:invariant}
The Weil pairing~\eqref{e:weil} depends only on the base change of $A$ to $\ovl{k}$. In particular, the Weil pairings induced on $A[2] \cong A^F[2]$ by all quadratic twists of $\lambda$ are the same. 
\end{rem}

From now until the end of this section we shall \textbf{fix the assumption} that $A$ has all of its $2$-torsion points defined over $k$. In this case we have an induced pairing 
\begin{equation}\label{e:weil-2}
\left<,\right>_{\lam}: H^1(k,A[2]) \times A[2] \lrar H^1(k,\mu_2)
\end{equation}
which by abuse of notation we shall denote by the same name. 
Let $M$ be a $2$-structure on $A$ (see Definition~\ref{d:2-struct}). For each $w \in M$ let $P_w \in A[2]$ be such that $\left<Q,P_w\right>_{\lam} = -1$ if and only if the image of $Q$ in $C_w/2C_w$ is non-trivial. It then follows from Definition~\ref{d:2-struct} that $\{P_w\}_{w \in M}$ forms a basis for $A[2]$. We will denote by $\{Q_w\}$ the dual basis of $\{P_w\}$ with respect to the Weil pairing. We note that by construction the image of $Q_w$ in $C_{w'}/2C_{w'}$ is non-trivial if and only if $w = w'$.

\begin{rem}
Since the Weil pairing is non-degenerate it follows that the association
$$ \bet \mapsto \left(\left<\bet,P_w\right>_{\lam}\right)_{w \in M} $$
determines an \textbf{isomorphism}
$$ H^1(k,A[2]) \x{\cong}{\lrar} H^1(k,\mu_2)^M $$
We may think of this isomorphism as a set of ``canonical coordinates'' on $H^1(k,A[2])$.
\end{rem}

\begin{rem}
While our notation for the group structure on $A[2]$ is additive, i.e., we write $P+Q$ for the sum of two points $P,Q \in A[2]$, our notation for the group structure on $\mu_2 = \{-1,1\}$ is \textbf{multiplicative}. For example, the linearity of $\left<,\right>_{\lam}$ in its left entry will be written as $\left<P+Q,R\right>_\lambda = \left<P,R\right>_{\lambda}\left<Q,R\right>_{\lambda}$. Similarly, the group operation of $H^1(k,A[2])$ will be written additively, while that of $H^1(k,\mu_2)$ multiplicatively.
\end{rem}

\begin{define}
We will denote by $\partial: A(k) \lrar H^1(k,A[2])$ the boundary map induced by the Kummer sequence of $A$. Similarly, for a quadratic extension $F/k$ we will denote by $\partial_{F}: A^F(k) \lrar H^1(k,A[2])$ the boundary map associated to the Kummer sequence of $A^F$, where we have implicitly identified $A^F[2]$ with $A[2]$ (see Remark~\ref{r:abuse}).
\end{define}

\begin{rem}
The bilinear map $(P,Q) \mapsto \left<\del(P),Q\right>_{\lam} \in H^1(k,\mu_2)$ is \textbf{not} symmetric in general. While this fact will not be used in this paper we note more precisely that
$$ \left<\del(P),Q\right>_{\lam}\left<\del(Q),P\right>_{\lam} = [\left<P,Q\right>_{\lam}], $$
where $[\left<P,Q\right>_{\lam}]$ denotes the image of $\left<P,Q\right>_{\lam}$ under the composed map $\mu_2 \lrar k^* \lrar k^*/(k^*)^2 \cong H^1(k,\mu_2)$. 
\end{rem}


For the purpose of the arguments in \S\ref{s:proof} we will need to establish some preliminary lemmas. The first one concerns the question of how the bilinear map $(P,Q) \mapsto \left<\del(P),Q\right>_{\lam}$ changes under quadratic twists.

\begin{lem}\label{l:twist}
Let $A$ be an abelian variety as above and let $P,Q \in A[2]$ be two $2$-torsion points. Let $F = k(\sqrt{a})$ be a quadratic extension. Then
$$  \left<\del_F(Q),P\right>_{\lam}\left<\del(Q),P\right>_{\lam}^{-1} =  \left\{\begin{matrix} 1 & \left<Q,P\right>_{\lam} = 1 \\ [a] & \left<Q,P\right>_{\lam} = -1 \\ \end{matrix}\right. $$
where $[a] \in k^*/(k^*)^2 \cong H^1(k,\mu_2)$ denotes the class of $a$ mod squares.
\end{lem}
\begin{proof}
Let $Z_Q \subseteq A$ be the finite subscheme determined by the condition $2x = Q$. Then $Z_Q$ carries a natural structure of an $A[2]$-torsor whose classifying element in $H^1(k,A[2])$ is given by $\del(Q)$. Given a point $x \in Z_Q(\ovl{k})$ we may represent $\del(Q)$ by the $1$-cocycle $\sig \mapsto \sig(x) - x$, and consequently represent $\left<\del(Q),P\right>_{\lam}$ by the $1$-cocycle $\sig \mapsto \left<\sig(x) - x,P\right>_{\lam} \in \mu_2$. Let $\Gam_k$ be the absolute Galois group of $k$ and 
let $\chi: \Gam_k \lrar \mu_2$ be the quadratic character associated with $F/k$. In light of Remark~\ref{r:invariant} we see that the class $\left<\del_F(Q),P \right>_{\lam}$ can be represented by the $1$-cocycle
$$ \sig \mapsto \left<\chi(\sig)\sig(x) - x,P\right>_{\lam} \in \mu_2 .$$
We may hence compute that
$$ \left<\chi(\sig)\sig(x) - x,P\right>_{\lam} \left<\sig(x) - x,P\right>_{\lam}^{-1} = \left<(\chi(\sig)-1)\sig(x),P\right>_{\lam} = \left\{\begin{matrix} 1 & \chi(\sig) = 1 \\ \left<Q,P\right>_{\lam} & \chi(\sig) = -1 \\ \end{matrix}\right. $$
This means that when $\left<P,Q\right>_{\lam} = 1$ the class $\left<\del_F(Q),P\right>_{\lam}\left<\del(Q),P\right>_{\lam}^{-1}$ vanishes, and when $\left<P,Q\right>_{\lam} = -1$ the class $\left<\del_F(Q),P\right>_{\lam}\left<\del(Q),P\right>_{\lam}^{-1}$ coincides with $[a]$, as desired.
\end{proof}

For a place $w \in \Om_k$ we will denote by $k^{\un}_w/k_w$ the maximal unramified extension of $k_w$. The following lemma concerns the Galois action on certain $4$-torsion points which are defined over extensions ramified at $w$.
\begin{lem}\label{l:p_w-1}
Let $w \in M$ be a place in the $2$-structure $M$ of $A$ and let $Q \in A[2]$ be a point whose image in $C_w/2C_w$ is non-trivial. Let $x \in A(\ovl{k})$ be a point such that $2x = Q$ and let $L_Q/k^{\un}_w$ be the minimal Galois extension such that all the points of $Z_Q$ are defined over $L_Q$. Then $\Gal(L_Q/k) \cong \ZZ/2$ and if $\sig \in \Gal(L_Q/k)$ is the non-trivial element then $\sig(x) = x + P_w$.
\end{lem}
\begin{proof}
After restricting $\del(Q)$ to $k^{\un}_w$ we obtain an injective homomorphism $\Gal(L_Q/k^{\un}_w) \lrar A[2]$. In particular, $L_Q$ is a finite abelian $2$-elementary extension of $k^{\un}_w$. Since $w$ is odd $L_w/k^{\un}_w$ is tamely ramified and hence cyclic, which means that $\Gal(L_Q/k^{\un}_w)$ is either $\ZZ/2$ or trivial. Since the image of $Q$ in $C_w/2C_w$ is non-trivial the latter is not possible and we may hence conclude that $\Gal(L_Q/k^{\un}_w)$. Let $\sig \in \Gal(L_Q/k^{\un}_w)$ be the non-trivial element.

Let $P \in A[2]$ be any $2$-torsion point whose image in $C_w/2C_w$ is trivial. It then follows from Hensel's lemma that there exists a $y \in A(k^{\un}_w)$ such that $2y = P$. Let us denote by
$$ \left<,\right>^4_{\lambda}: A[4] \times A[4] \lrar \GG_m[4] $$
the Weil pairing on the $4$-torsion. Then by the compatibility property of the Weil pairings we have
$$ \left<\sig(x)-x,P\right>_{\lambda} = \left<\sig(x)-x,y\right>^4_{\lambda} = \left<\sig(x),y\right>^4_{\lambda}\left[\left<x,y\right>^4_{\lambda}\right]^{-1} = 1 $$
where the last equality holds since $\sig(y) = y$, the Weil pairing is Galois invariant, and $\GG_m[4]$ is fixed by $\sig$. It follows that the $2$-torsion point $\sig(x)-x$ is orthogonal to every $2$-torsion point whose $w$-reduction lies on the identity component. The only two $2$-torsion points which have this orthogonality property are $0$ and $P_w$. The former option is not possible since $x$ is not defined over $k$ and hence we may conclude that $\sig(x) - x = P_w$, as desired.
\end{proof}

We now explore two corollaries of Lemma~\ref{l:p_w-1}.
\begin{cor}\label{c:p_w-2}
Let $w \in M$ be a place in the $2$-structure $M$ of $A$, let $L/k$ be a non-trivial quadratic extension which is ramified at $w$ and let $w'$ be the unique place of $L$ lying above $w$. Let $A_L = A \otimes_{k} L$ be the base change of $A$ to $L$ and let $C_{w'}$ be the group of components of the reduction of the Néron model of $A_L$ at $w'$. Then $2$-part of $C_{w'}$ is cyclic of order $4$ and the induced action of $\Gal(L/k)$ on $C_{w'}/4C_{w'} \cong \ZZ/4$ is trivial.
\end{cor}
\begin{proof}
Since the reduction of $A$ at $w$ is multiplicative the reduction of $A_L$ at $w'$ is multiplicative and $C_w$ naturally embeds in $C_{w'}$. Since $w$ is part of a $2$-structure the $2$-part $C_w$ is cyclic of order $2$. Since all the $2$-torsion points of $A$ are defined over $k$ the group $C_{w'}$ cannot have $2$-torsion elements which do not come from $C_w$, and hence $C_{w'}$ is cyclic as well. Let us now show that the $2$-part of $C_{w'}$ is of order $4$. Let $L' = L\cdot k^{\un}_w$ be the compositum of $L$ and the maximal unramified extension $k^{\un}_w$. Since $L$ is ramified we see that $L'/k^{\un}_w$ is a non-trivial quadratic extension. Since $L'/L$ is unramified we may identify the Néron model $\cA_{L'}$ of $A_{L'}$ with the base change of the Néron model $\cA_L$ from $\OO_L$ to $\OO_{L'}$. In particular, we may identify the special fiber of $\cA_{L'}$ with the base change of the reduction of $\cA_L$ at $w$ from $\FF_w$ to $\ovl{\FF}_w$, and consequently identify the group of components of $\cA_{L'}$ with $C_{w'}$.

Since $w$ is odd the field $k^{\un}_w$ has a unique quadratic extension. In particular, if $Q \in A[2]$ is a point whose image in $C_w/2C_w$ is non-trivial then $L'$ must coincide with the splitting field $L_Q$ appearing in Lemma~\ref{l:p_w-1}. In particular, there exists an $x \in A(L')$ such that $2x = Q$ and $\sig(x) - x = P_w$, where $\sig$ is the non-trivial element of $\Gal(L'/k^{\un}_w)$. Now the reduction of $Q \in A(L')$ lies on a component of $C_{w'}$ of exact order $2$ (since this is the case for $C_w$ and the map $C_w \hrar C_{w'}$ induced by base change is injective), and hence the reduction of $x$ lies on a component $X \in C_{w'}$ of order exactly $4$. In particular, the order of the $2$-part of the cyclic group $C_{w'}$ is at least $4$. To show that it is exactly $4$, let by assume by contradiction that there exists a component $X' \in C_{w'}$ such that $2X' = X$. Multiplication by $2$ then induces a map of $\FF_w$-schemes $X'\lrar X$ which is an étale covering and a torsor under $A[2]$. It then follows that there exists a $\ovl{y} \in X'(\ovl{F}_w)$ such that $2\ovl{y} = \ovl{x}$, where $\ovl{x} \in X(\ovl{\FF}_w)$ is the reduction of $x$. By Hensel's lemma there exists a $y \in A(L')$ such that $2y = x$ and the reduction of $y$ is $\ovl{y}$. By Lemma~\ref{l:p_w-1} we now get that $2(\sig(y) - y) = \sig(x) - x = P_w$ and hence $z = \sig(y) - y$ is a torsion point of exact order $4$. Since $\sig(z) = -z$ and $z$ is of exact order $4$ it follows that $z$ is not defined over $k$. But this is now a contradiction since $P_w$ maps to the trivial element of $C_w/2C_w$ (since $\left<P_w,P_w\right>_{\lam} = 1 \in \mu_2$) and hence by Hensel's Lemma all the $4$-torsion points $z$ such that $2z = P_w$ are defined over $k^{\un}_w$. We may hence conclude that the $2$-part of $C_{w'}$ is a cyclic group of order $4$, as desired.

Since $w$ is odd, one consequence of the fact that $2$-part of $C_{w'}$ is cyclic of order $4$ is that the points of $A(L')$ whose image in $C_{w'}/4C_{w'}$ is trivial can be characterized as those points which are infinitely $2$-divisible in $A(L')$. One may hence identify $C_{w'}/4C_{w'}$ with the quotient of $A(L')$ by its maximal $2$-divisible subgroup, and use this identification to read off the action of $\Gal(L'/k^{\un}_w)$ on $C_{w'}/4C_{w'}$. In particular, since the reduction of the point $x \in A(L')$ from the discussion above lies on a component of exact order $4$ and the reduction of $\sig(x) = x + P_w$ lies on the same component this action must be trivial.
\end{proof}

\begin{cor}\label{c:p_w-3}
Let $Q,Q' \in A[2]$ be two points. Then $\alp := \left<\del(Q),Q'\right>_{\lam}$ is ramified at $w$ if and only if the images of both $Q$ and $Q'$ in $C_w/2C_w$ are non-trivial. 
\end{cor}
\begin{proof}
If $Q$ reduces to the identity of $C_w/2C_w$ then the entire class $\del(Q) \in H^1(k,A[2])$ is unramified. We may hence assume that $Q$ reduces to the non-trivial element of $C_w/2C_w$. Let $Z_Q \subseteq A$ be the finite subscheme determined by the condition $2x = Q$ and let $L/k^{\un}_w$ be the minimal Galois extension such that all the points of $Z_Q$ are defined over $L$. Let $x \in Z_Q(L)$ be any point. By Lemma~\ref{l:p_w-1} we know that $G := \Gal(L/k^{\un}_w)$ is isomorphic to $\ZZ/2$ and that if $\sig \in G$ denotes the non-trivial element then $\sig(x) -x = P_w$.
Since $\del(Q)$ vanishes when restricted to $L$ the same holds for $\alp$ and by the inflation-restriction exact sequence the element $\alp|_{k^{\un}_w}$ comes from an element $\ovl{\alp} \in H^1(\Gal(L/k^{\un}_w,\mu_2) = H^1(G,\mu_2)$, which in turn can be written as a homomorphism
$$ \ovl{\alp}: G \lrar \mu_2 .$$
Furthermore, as in the proof of~\ref{l:p_w-1} the value $\ovl{\alp}(\sig)$ is given by the explicit formula
$$ \ovl{\alp}(\sig) = \left<\sig(x)-x,Q'\right>_{\lam} = \left<P_w,Q'\right>.$$
By the definition of $P_w$ we now get that $\alp|_{k^{\un}_w}$ is trivial if and only if $Q'$ reduces to the identity of $C_w/2C_w$, as desired.
\end{proof}

\subsection{The Cassels-Tate pairing}\label{s:CT}

Let $A$ be an abelian variety over a number field $k$ with dual abelian variety $\hat{A}$. Recall the \textbf{Cassels-Tate pairing}
$$ \langle \alp,\bet\rangle^{\CT}: \Sha(A) \times \Sha(\hat{A}) \lrar \QQ/\ZZ $$
whose kernel on either side is the corresponding group of divisible elements (which is believed to be trivial by the Tate-Shafarevich conjecture).

There are many equivalent ways of defining the Cassels-Tate pairing. In this paper it will be useful to have an explicit description of it via evaluation of Brauer element. We will hence recall the following definition, which is essentially the ``homogeneous space definition'' appearing in~\cite{PS99}. Let $\alp \in \Sha(A), \bet \in \Sha(\hat{A})$ be elements, we may describe the Cassels-Tate pairing $\langle \alp,\bet\rangle^{\CT}$ as follows. Let $Y_{\alp}$ be the torsor under $A$ classified by $\alp$. Since $\alp$ belongs to $\Sha(A)$ we have that $Y_\alp(\AA_k) \neq \emptyset$. The Galois module $\Pic^0(\ovl{Y}_{\alp})$ is canonically isomorphic to $\hat{A}(\ovl{k})$. 
Let $\Br_1(Y_{\alp}) = \Ker[\Br(Y_{\alp}) \lrar \Br(\ovl{Y}_{\alp})]$ be the algebraic Brauer group of $Y_{\alp}$. The Hochschild-Serre spectral sequence yields an isomorphism
$$ \Br_1(Y_{\alp})/\Br(k) \x{\cong}{\lrar} H^1(k,\Pic(\ovl{Y}_{\alp})) .$$
Let
$$ B_{\alp}:H^1(k,\hat{A}) \x{\cong}{\lrar} H^1(k,\Pic^0(\ovl{Y}_{\alp})) \lrar H^1(k,\Pic(\ovl{Y}_{\alp})) \cong \Br_1(Y_{\alp})/\Br(k) $$
denote the composed map. 
\begin{define}\label{d:CT}
Let $B \in \Br(Y_\alp)$ be an element whose class in $\Br(Y_\alp)/\Br(k)$ is $B_{\alp}(\bet)$ and let $(x_v) \in Y_\alp(\AA_k)$ be an adelic point. The the Cassels-Tate pairing of $\alp$ and $\bet$ is given by
$$ \langle \alp,\bet \rangle^{\CT} := \sum_{v \in \Om_k} B(x_v) \in \QQ/\ZZ $$
\end{define}

Given a principal polarization $\lam: A \x{\cong}{\lrar} \hat{A}$ we obtain an isomorphism $\Sha(A) \cong \Sha(\hat{A})$ and hence a self pairing 
$$ \left<,\right>^{\CT}_{\lambda}: \Sha(A) \times \Sha(A) \lrar \QQ/\ZZ $$

\begin{rem}\label{r:alternating}
The pairing $\left<,\right>^{\CT}_{\lambda}$ is not alternating in general. However, as is shown in~\cite{PR}, it is the case that $\left<,\right>^{\CT}_{\lambda}$ is alternating when $\lam$ is induced by a symmetric line bundle on $A$. The obstruction to realizing $\lam$ as a symmetric line bundle is an element $c_{\lam} \in H^1(k,A[2])$, which vanishes, for example, when the Galois action on $A[2]$ is trivial, see~\cite[Lemma 5.1]{HS15}. In particular, in all the cases considered in this paper the Cassels-Tate pairing associated to a principal polarization will always be alternating. 
\end{rem}

\subsection{Kummer varieties}\label{s:kummer}

In this section we will review some basic notions and constructions concerning Kummer varieties. Let $A$ be an abelian variety over $k$ (not necessarily principally polarized) of dimension $g \geq 2$. Let $\alp \in H^1(k,A[2])$ be a class and let $Y_{\alp}$ be the associated $2$-covering of $A$. Then $Y_\alp$ is equipped with a natural action of $A[2]$ and the base change of $Y_\alp$ to the algebraic closure of $k$ is $A[2]$-equivariantly isomorphic to the base change of $A$. More precisely, the class $\alp$ determines a distinguished Galois invariant subset of $A[2]$-equivariant isomorphisms $\Psi_\alp \subseteq \Iso_{A[2]}(\ovl{Y}_\alp,\ovl{A})$ which is a torsor under $A[2]$ with class $\alp$ (where $A[2]$ acts on $\Iso_{A[2]}(\ovl{Y}_\alp,\ovl{A})$ via post-composition). Using any one of the isomorphisms $\psi \in \Psi_\alp$ we may transport the antipodal involution $\iota_A: A \lrar A$ to an involution $\iota_{\psi}: \ovl{Y} \lrar \ovl{Y}$. Since $\iota_A$ commutes with translations by $A[2]$ it follows that $\iota_{\psi}$ is independent of $\psi$, and hence determines a Galois invariant automorphism $\ovl{Y}_\alp \lrar \ovl{Y}_\alp$. By classical Galois descent we may realize this Galois invariant automorphism uniquely as an automorphism $\iota_{Y_\alp}: Y_{\alp} \lrar Y_{\alp}$ defined over $k$. 

Let $Z_\alp \subseteq Y_\alp$ denote the fixed locus of $\iota_{Y_\alp}$ (considered as a $0$-dimensional subscheme). We note that the points of $Z_\alp(\ovl{k})$ are mapped to $A[2]$ by any of the isomorphisms $\psi \in \Psi_\alp$, and the Galois invariant collection of isomorphisms $\{\psi|_{\ovl{Z}_\alp} | \psi \in \Psi_\alp\}$ exhibits $Z_\alp$ as a torsor under $A[2]$ with class $\alp$. The quotient $(Y_\alp)/\iota_{Y_\alp}$ has $Z_\alp$ as its singular locus and this singularity can be resolved by a single blow-up. Alternatively, one can first consider the blow-up $\wtl{Y}_\alp$ of $Y_\alp$ at $Z_\alp$, and then take the quotient of $\wtl{Y}_\alp$ by the induced involution $\iota_{\wtl{Y}_\alp}$. A local calculation then shows that $\wtl{Y}_\alp/\iota_{\wtl{Y}_\alp}$ is \textbf{smooth}.

\begin{define}
The \textbf{Kummer variety} associated to $Y_\alp$ is the variety $$\Kum(Y_\alp) = \wtl{Y}_\alp/\iota_{\wtl{Y}_\alp}$$
\end{define}

Let now $X_\alp = \Kum(Y_\alp)$ be the Kummer variety of $Y_\alp$. We will denote by $D_\alp \subseteq \wtl{Y}_\alp$ the exceptional divisor. Since the action of $\iota_{\wtl{Y}_\alp}$ on $D_\alp$ is constant we will abuse notation and denote the image of $D_\alp$ in $X_\alp$ by the same name. Let us denote by $U_\alp = \wtl{Y}_\alp \setminus D_\alp$ and $W_\alp = X_\alp \setminus D_\alp$, so that the quotient map $\wtl{Y}_\alp \lrar X_\alp$ restricts to an \'etale $2$-covering $p_\alp:U_\alp \lrar W_\alp$. Let $\iota_{U_\alp}: U_\alp \lrar U_\alp$ denote the restriction of $\iota_{Y_\alp}$. We note that we may also identify $U_\alp$ with the complement of the $0$-dimensional scheme $Z_\alp$ in $Y_\alp$. Since $\dim Y_\alp \geq 2$ we may consequently identify $H^1(\ovl{U}_\alp,\QQ/\ZZ(1))$ with $H^1(\ovl{Y}_\alp,\QQ/\ZZ(1))$ and $H^1(\ovl{U}_\alp,\QQ/\ZZ(1))^{\iota_{U_\alp}}$ with $\hat{A}[2]$. Now since $H^2(\left<\iota_{U_\alp}\right>,H^0(\ovl{U}_\alp,\QQ/\ZZ(1))) = 0$ the Hochshild-Serre spectral sequence now yields a short exact sequence of Galois modules
\begin{equation}\label{e:hs}
0 \lrar \mu_2 \x{\iota}{\lrar} H^1(\ovl{W}_\alp,\QQ/\ZZ(1)) \x{p_\alp^*}{\lrar} \hat{A}[2] \lrar 0 
\end{equation}
where the image of $\iota$ is spanned the element $[p_\alp] \in H^1(\ovl{W}_\alp,\mu_2) \subseteq H^1(\ovl{W}_\alp,\QQ/\ZZ(1))$ which classifies the $2$-covering $p_\alp:U_\alp \lrar W_\alp$. 

Our next goal is to describe the Galois module $H^1(\ovl{W}_\alp,\QQ/\ZZ(1))$ in more explicit terms. For this it will be convenient to use the following terminology. Let us say that a map of schemes $L: \ovl{Z}_\alp \lrar \mu_2$ is \textbf{affine-linear} if the there exists a $Q \in \hat{A}[2]$ such that for every geometric point $x \in Z_\alp(\ovl{k})$ and every $P \in A[2]$ we have $L(Px) = \left<Q,P\right> \cdot L(x)$ (here the notation $Px$ denotes the action of $A[2]$ on its torsor $Z_\alp$). We will refer to $Q$ as the \textbf{homogeneous part} of $L$. We note that $Q$ is uniquely determined by $L$. We will denote by $\Aff(\ovl{Z}_\alp,\mu_2)$ the abelian group of affine-linear maps (under pointwise multiplication). The action of $\Gam_k$ on $\ovl{Z}_\alp$ induces an action on $\Aff(\ovl{Z}_\alp,\mu_2)$ by pre-composition and will consequently consider $\Aff(\ovl{Z}_\alp,\mu_2)$ as a Galois module. The map 
$$ h_\alp: \Aff(\ovl{Z}_\alp,\mu_2) \lrar \hat{A}[2] $$ 
which assigns to each affine-linear map its homogeneous part is then homomorphism of Galois modules. The following lemma is a variant of~\cite[Proposition 2.3]{SZ16} and is essentially reformulated to make the Galois action more apparent. Here we consider $\Aff(\ovl{Z}_\alp,\mu_2)$ as a Galois submodule of $H^0(\ovl{Z}_\alp,\QQ/\ZZ)$, by identifying elements of the latter with set theoretic functions $Z_\alp(\ovl{k}) \lrar \QQ/\ZZ$ and using the embedding $\mu_2 \cong \frac{1}{2}\ZZ/\ZZ \hrar \QQ/\ZZ$.

\begin{lem}[{cf.~\cite[Proposition 2.3]{SZ16}}]\label{l:W}
The residue map $H^1(\ovl{W}_\alp,\QQ/\ZZ(1)) \lrar H^0(\ovl{D}_\alp,\QQ/\ZZ)$ is injective and its image coincides with $\Aff(\ovl{Z}_\alp,\mu_2)$. Furthermore, the resulting composed map $H^1(\ovl{W}_\alp,\QQ/\ZZ(1)) \lrar \Aff(\ovl{Z}_\alp,\mu_2) \x{h_\alp}{\lrar} \hat{A}[2]$ coincides with the map $p_\alp^*$ appearing in~\eqref{e:hs}.
\end{lem}
\begin{proof}
For the purpose of this lemma we may as well extend our scalars to the algebraic closure. We may hence assume without loss of generality that $\alp = 0$ (i.e., that $X = \Kum(A)$) and that the Galois action on $A[2]$ is constant. We will consequently write $A$ instead of $Y$, $A[2]$ instead of $Z_\alp$, $W$ instead of $W_\alp$ and $U$ instead of $U_\alp$. Let $Q \in \hat{A}[2]$ be a non-zero element and let $f_Q:\hat{A} \lrar B$ be an isogeny of abelian varieties whose kernel is spanned by $Q$. Then the dual isogeny $\hat{f}_Q: \hat{B} \lrar A$ is the unramified $2$-covering classified by $Q \in \hat{A}[2] \cong H^1(\ovl{A},\mu_2)$. Let $C$ be the variety obtained from $\hat{B}$ by blowing up the pre-image $M = \hat{q}^{-1}(A[2])$ of the $2$-torsion points of $A$. Given another point $P \in A[2]$ we may consider the map $\wtl{f}_{Q,P}: \hat{B} \lrar A$ given by $\wtl{f}_P(x) = \wtl{q}(x) + P$. For each $P \in A[2]$ the map $\wtl{f}_{Q,P}$ is an unramified $2$-covering sending $M$ to $A[2]$ and hence induces an unramified $2$-covering 
$$ q_{Q,P}:C \lrar \wtl{A} ,$$ 
where $\wtl{A}$ denotes the blow-up of $A$ at $A[2]$. Consider the automorphisms $\iota_C: C \lrar C$ and $\iota_{\wtl{A}}: \wtl{A} \lrar \wtl{A}$ induced by the respective antipodal involutions. Since $\wtl{q}_P$ commutes with the respective antipodal involutions the same holds for $r$. We then obtained an induced (ramified) $2$-covering map between smooth varieties
$$ q_{Q,P}: C/\iota_C \lrar \wtl{A}/\iota_{\wtl{A}} = \Kum(A) .$$
Note that $q_{Q,P}$ is unramified over the complement $W \subseteq \Kum(A)$: indeed, any geometric point $x \in W(\ovl{k})$ has two points lying above it in $\wtl{A}$, and hence four points lying above it in $C$, which must give two distinct points in $C/\iota_C$. The pullback of $q_{Q,P}$ to $W$ hence determines an unramified $2$-covering
$$ r_{Q,P}: V \lrar W $$
which determines an element $[r_{Q,P}] \in H^1(W,\mu_2) \subseteq H^1(\ovl{W},\QQ/\ZZ(1))$. Now consider the map $\hat{f}_{Q,P}': \hat{B} \setminus M \lrar U$ obtained by restricting the domain and codomain of $\hat{f}_{Q,P}$. Then $\hat{f}_{Q,P}'$ is an étale $2$-covering and the induced map $\hat{B} \setminus M \lrar V \times_{W} U$ is étale of degree $1$, and hence an isomorphism. It follows that $p^*[r_{Q,P}] = [\hat{f}_{Q,P}'] \in H^1(U,\mu_2)$. On the other hand, the inclusion $U \subseteq A$ as the complement of $A[2]$ induces an isomorphism $H^1(\ovl{U},\mu_2) \cong H^1(\ovl{A},\mu_2) \cong \hat{A}[2]$ which identifies the image of $[\hat{f}_{Q,P}']$ in $H^1(\ovl{U},\mu_2)$ with $Q \in \hat{A}[2]$. Finally, a direct examination verifies that $r_{Q,P}$ is ramified at the exceptional divisor $D_x \subseteq X \setminus W$ corresponding to $x \in A[2]$ if and only if $x$ is not in the image of ${\hat{f}_{Q,P}}|_{\hat{B}[2]}: \hat{B}[2] \lrar A[2]$. By the compatibility of the Weil pairing with duality of isogenies we see that the image of $\hat{f}_{Q,0}|_{\hat{B}[2]}$ consists of exactly those $x \in A[2]$ such that the Weil pairing $\left<x,Q\right>$ is trivial, and hence the image of $\hat{f}_{Q,P}$ consists of those $x \in A[2]$ such that $\left<x-P,Q\right>$ is trivial. It then follows that the residue $\res_{D}([r_{Q,P}]) \in H^0(A[2],\QQ/\ZZ) \cong (\QQ/\ZZ)^{A[2]}$ can be identified with the affine-linear function $L_{Q,P}(x) = \left<x,Q\right>\left<P,Q\right>$. 

We note that by varying $P$ we obtain in this way for each non-zero $Q \in A[2]$ at least two different elements of $H^1(\ovl{W},\mu_2)$ whose image in $H^1(\ovl{U},\mu_2) \cong \hat{A}[2]$ is $Q$. By the short exact sequence~\eqref{e:hs} we have thus covered \textbf{all} elements of $H^1(\ovl{W},\QQ/\ZZ(1))$ whose image in $\hat{A}[2]$ is non-trivial. On the other hand, the only non-trivial element of the kernel $H^1(\ovl{W},\QQ/\ZZ(1)) \lrar \hat{A}[2]$ is the one classifying the covering $\ovl{U} \lrar \ovl{W}$, which is ramified at all the components $D_x$, and whose residue hence corresponds to the constant function $A[2] \lrar \QQ/\ZZ$ with value $1/2$. Finally, the trivial element has trivial residue, which corresponds to the constant function $A[2] \lrar \QQ/\ZZ$ with value $0$. This concludes the enumeration of all element of $H^1(\ovl{W},\QQ/\ZZ(1))$, and so the proof is complete.   
\end{proof}

By Lemma~\ref{l:W} the residue map induces an isomorphism of Galois modules $H^1(\ovl{W}_\alp,\QQ/\ZZ(1)) \cong \Aff(\ovl{Z}_\alp,\mu_2)$, and we may rewrite the short exact sequence~\eqref{e:hs} as
\begin{equation}\label{e:sko-2}
\xymatrix{
0 \ar[r] & \mu_2 \ar^-{\iota}[r] & \Aff(\ovl{Z}_{\alp},\mu_2) \ar^-{h_\alp}[r] & \hat{A}[2] \ar[r] & 0 \\
}
\end{equation}
where $\iota:\mu_2 \hrar \Aff(\ovl{Z}_{\alp},\mu_2)$ is the inclusion of constant affine-linear functions. We note that since $\ovl{U}_\alp \cong \ovl{Y}_{\alp} \setminus \ovl{Z}_\alp \cong \ovl{A} \setminus A[2]$ has no non-constant invertible functions the same holds for $\ovl{W}_\alp$ and so we have a canonical isomorphism of Galois modules $H^1(\ovl{W}_\alp,\QQ/\ZZ(1)) \cong \Pic(\ovl{W}_\alp)_{\tor}$. We note that the injectivity of the residue map $H^1(\ovl{W}_\alp,\QQ/\ZZ(1)) \lrar H^0(\ovl{D}_\alp,\QQ/\ZZ)$ implies, in particular, that $H^1(\ovl{X}_\alp,\QQ/\ZZ(1)) = 0$ and hence that $\Pic(\ovl{X}_\alp)$ is \textbf{torsion free}.

\begin{rem}
Consider the pullback map $\Pic(\ovl{X}_\alp) \lrar \Pic(\ovl{W}_\alp)$ on geometric Picard groups. The inverse image $\Pi \subseteq \Pic(\ovl{X}_\alp)$ of the torsion subgroup $\Pic(\ovl{W}_\alp)_{\tor} \subseteq \Pic(\ovl{W}_\alp)$ is called the \textbf{Kummer lattice} in~\cite{SZ16}. Given an affine-linear map $L: \ovl{Z}_\alp \lrar \mu_2$, we may realize the corresponding element of $\Pic(\ovl{W}_\alp)_{\tor}$ as a $2$-covering of $\ovl{W}_\alp$. This $2$-covering extends to a $2$-covering of $\ovl{X}_\alp$ which is ramified along $D_x$ if and only if $L(x) = -1$. It then follows that there exists a class $E_L \in \Pic(\ovl{X}_\alp)$ such that 
$$ 2E_L = \sum_{x \in Z_\alp(\ovl{k}) | L(x) = -1} [D_x] $$ 
and the image of $E_L$ in $\Pic(\ovl{W}_\alp)_{\tor}$ is $L$. In particular, the Kummer lattice is generated over $\Pi_0$ by the classes $E_L$. This description of the Kummer lattice was established by Nikulin (\cite{Ni75}) in the case of Kummer surfaces and extended to general Kummer varieties by Skorobogatov and Zarhin in~\cite{SZ16}.
\end{rem}

From now until the rest of this section we \textbf{fix the assumption} that the Galois action on $A[2]$ is \textbf{constant}. Recall from \S\ref{s:CT} that we have a homomorphism
$$ B_\alp: H^1(k,\hat{A}) \lrar \Br(Y_\alp)/\Br(k) $$
which can be used to define the Cassels-Tate pairing between $\alp$ and $\bet$.
It will be useful to consider similar types of Brauer elements on $W_\alp$. Using the map $H^1(k,\Pic(\ovl{W}_\alp)) \lrar \Br(W_\alp)/\Br(k)$ furnished by the Hochschild-Serre spectral sequence, the map
$$\Aff(\ovl{Z}_{\alp},\mu_2) \cong \Pic(\ovl{W}_\alp)_{\tor} \lrar \Pic(\ovl{W}_\alp) $$ 
determines a map 
\begin{equation}\label{e:C-alp} 
C_\alp: H^1(k,\Aff(\ovl{Z}_{\alp},\mu_2)) \lrar \Br(W_\alp)/\Br(k) 
\end{equation}
which fits into a commutative square
$$ \xymatrix{
H^1(k, \Aff(\ovl{Z}_{\alp},\mu_2)) \ar^{C_\alp}[d]\ar^-{(h_\alp)_*}[r] & H^1(k,\hat{A}[2]) \ar^{B_\alp}[d] \\
\Br(W_\alp)/\Br(k) \ar[r] & \Br(Y_\alp)/\Br(k) \\
}$$
where the bottom horizontal map is induced by the composition of $p_\alp^*: \Br(W_\alp) \lrar \Br(U_\alp)$ and the natural isomorphism $\Br(U_\alp) \cong \Br(Y_\alp)$. 
It will be useful to recall the following general construction:

\begin{const}\label{c:galois}
Let $G$ be a group acting on an abelian group $M$ and let $f: \Gam_k \lrar G$ be a homorphism, through which we can consider $M$ as a Galois module. Let $x \in H^1(k,M)$ be an element. Then $x$ classifies a torsor $Z_x$ under $M$, and the Galois action on $Z_x(\ovl{k})$ is via the semi-direct product $M \rtimes G$. The kernel of the resulting homomorphism $\Gam_k \lrar M \rtimes G$ is a normal subgroup $\Gam_x \subseteq \Gam_k$ and we will refer to the corresponding normal extension $k_x/k$ as the \textbf{splitting field} of $x$. We then obtain a natural injective homomorphism $\ovl{x}: \Gal(k_x/k) \lrar M \rtimes G$. We note that the field $k_x$ contains the field $k_f$ associated to the kernel of $f: \Gam_k \lrar G$ and the restriction of $\ovl{x}$ to $\Gal(k_x/k_f)$ lands in $M$. 
Finally, the homomorphism $f: \Gam_k \lrar G$ descends to an injective homomorphism $\ovl{f}: \Gal(k_f/k) \lrar G$, and we obtain a commutative diagram with exact rows and injective vertical maps
\begin{equation}
\xymatrix@C=1.1em{
1 \ar[r] & \Gal(k_x/k_f) \ar[r]\ar@{^(->}^{\ovl{x}|_{k_f}}[d] & \Gal(k_x/k) \ar[r]\ar@{^(->}^{\ovl{x}}[d] & \Gal(k_f/k) \ar[r]\ar@{^(->}^{\ovl{f}}[d] & 1 \\
1 \ar[r] & M \ar[r] & M \rtimes G \ar[r] & G \ar[r] & 1 \\
}
\end{equation}
Here, the notation $\ovl{x}|_{k_f}$ is meant to suggest that we can think of the left most vertical homomorphism as the one associated to the class $x|_{k_f} \in H^1(k_f,M)$ by the same construction.
\end{const}


Now let $\bet \in H^1(k,\hat{A}[2])$ be an element and suppose that $\theta \in H^1(k, \Aff(\ovl{Z}_{\alp},\mu_2))$ is such that $(h_{\alp})_*(\theta) = \bet \in H^1(k,\hat{A}[2])$. Applying Construction~\ref{c:galois} to $\alp \in H^1(k,A[2])$ we obtain an injective homomorphism $\ovl{\alp}:\Gal(k_\alp/k) \lrar A[2]$. The action of $\Gam_k$ on $\Aff(\ovl{Z}_{\alp},\mu_2))$ is given via the composition 
$$ \Gam_k \lrar \Gal(k_\alp/k) \x{\ovl{\alp}}{\lrar} A[2] ,$$ 
where $A[2]$ acts on $\Aff(\ovl{Z}_{\alp},\mu_2))$ via pre-composition. Applying Construction~\ref{c:galois} again we obtain a commutative diagram with exact rows and injective vertical maps
\begin{equation}\label{e:split}
\xymatrix@C=1.1em{
1 \ar[r] & \Gal(k_{\theta}/k_\alp) \ar[r]\ar@{^(->}^{\ovl{\theta}|_{k_\alp}}[d] & \Gal(k_{\theta}/k) \ar[r]\ar@{^(->}^{\ovl{\theta}}[d] & \Gal(k_\alp/k) \ar[r]\ar@{^(->}^{\ovl{\alp}}[d] & 1 \\
1 \ar[r] & \Aff(\ovl{Z}_{\alp},\mu_2) \ar[r] & \Aff(\ovl{Z}_{\alp},\mu_2) \rtimes A[2] \ar[r] & A[2] \ar[r] & 1 \\
}
\end{equation}
Here we may also consider $\ovl{\theta}|_{k_\alp}$ as the homomorphisms associated to $\theta|_{k_\alp} \in H^1(k_\alp,\Aff(\ovl{Z}_{\alp},\mu_2))$ by Construction~\ref{c:galois}. Let $k_{\alp,\bet}$ be the compositum of $k_\alp$ and $k_{\bet}$. By the naturality of Construction~\ref{c:galois} we also see that the Galois group $\Gal(k_{\theta}/k_\alp)$ sits in a commutative diagram with exact rows and injective vertical maps:
\begin{equation}\label{e:split-2}
\xymatrix@C=1.1em{
1 \ar[r] & \Gal(k_{\theta}/k_{\alp,\bet}) \ar[r]\ar@{^(->}^{\ovl{\theta}|_{k_{\alp,\bet}}}[d] & \Gal(k_{\theta}/k_{\alp}) \ar[r]\ar@{^(->}^{\ovl{\theta}|_{k_\alp}}[d] & \Gal(k_{\alp,\bet}/k_\alp) \ar[r]\ar@{^(->}^{\ovl{\bet}|_{k_\alp}}[d] & 1 \\
1 \ar[r] & \mu_2 \ar[r] & \Aff(\ovl{Z}_{\alp},\mu_2) \ar[r] & \hat{A}[2] \ar[r] & 1 \\
}
\end{equation}  
where the bottom row is the sequence~\eqref{e:sko-2}. Here $\ovl{\theta}|_{k_\alp},\ovl{\bet}|_{k_\alp}$ are the homomorphisms associated to the classes $\theta|_{k_\alp} \in H^1(k_\alp,\Aff(\ovl{Z}_{\alp},\mu_2)), \bet|_{k_\alp} \in H^1(k_\alp,\hat{A}[2])$ respectively by Construction~\ref{c:galois}, and $\ovl{\theta}|_{k_{\alp,\bet}}$ is the homomorphism associated to the class $\theta|_{k_{\alp,\bet}}$, when we consider it as belonging to $H^1(k_{\alp,\bet},\mu_2)$ (indeed, the homorphism $\iota_*:H^1(k_{\alp,\bet},\mu_2) \lrar H^1(k_{\alp,\bet},\Aff(\ovl{Z}_{\alp},\mu_2))$ induced by the sequence~\eqref{e:sko-2} is injective and its image contains the restricted class $\theta|_{k_{\alp,\bet}}$).

%
%
%
%

The following proposition plays a key role in the analysis of the behavior of the Cassels-Tate pairing under quadratic twists (see Proposition~\ref{p:CT-twist}):
\begin{prop}\label{p:descent}\
\begin{enumerate}[(1)]
\item
An element $\bet \in H^1(k,\hat{A}[2])$ can be lifted to an element $\theta \in H^1(k, \Aff(\ovl{Z}_{\alp},\mu_2))$ if and only if $\alp \cup \bet = 1 \in H^2(k,\mu_2)$. Furthermore, if $S$ is a set of places which contains a set of generators for the class group of $k$ and $\alp,\bet$ are unramified outside $S$ then $\theta$ can be chosen so that the splitting field $k_{\theta}$ is unramified outside $S$.
\item
The image of the $\res_{D_\alp}(C_{\alp}(\eps)) \in H^1(D_\alp,\QQ/\ZZ)$ in $H^1(D_\alp \otimes_k k_{\alp,\bet},\QQ/\ZZ)$ is constant and comes from the element $u_{\theta} \in H^1(k_{\alp,\bet},\ZZ/2)$ which classifies the at most quadratic extension $k_{\theta}/k_{\alp,\bet}$.
\end{enumerate}
\end{prop}
\begin{proof}
We begin by proving (1). For the elements $\alp \in H^1(k,A[2])$ and $\bet \in H^1(k,\hat{A}[2])$ we have corresponding homomorphisms $\ovl{\alp}: \Gal(k_\alp/k) \lrar A[2]$ and $\ovl{\bet}: \Gal(k_\bet/k) \lrar \hat{A}[2]$ as in Construction~\ref{c:galois}. By abuse of notation we will also denote by $\ovl{\alp}$ the composition of $\ovl{\alp}$ with the canonical projection $\Gam_k \lrar \Gal(k_\alp/k)$, and similarly for $\ovl{\bet}$. Consider the exact sequence 
\begin{equation}\label{e:aff}
H^1(k,\mu_2) \lrar H^1(k,\Aff(\ovl{Z}_{\alp},\mu_2)) \x{(h_\alp)_*}{\lrar} H^1(k,\hat{A}[2]) \x{\partial}{\lrar} H^2(k,\mu_2)
\end{equation}
associated to the short exact sequence~\eqref{e:sko-2}. We note that by choosing a base point $x_0 \in Z_\alp(\ovl{k})$ we may identify $Z_\alp(\ovl{k}) \cong A[2]$ and consequently identify each affine-linear map $L: Z_\alp(\ovl{k}) \lrar \mu_2$ with a map $A[2] \lrar \mu_2$ of the form $P \mapsto \eps \cdot \left<Q,P\right>$ for some $Q \in \hat{A}[2]$ and $\eps \in \mu_2$. The association $L \mapsto (\eps,Q)$ then identifies the underlying abelian group of $\Aff(\ovl{Z}_{\alp},\mu_2)$ with the abelian group $\mu_2 \oplus \hat{A}[2]$ and identifies the corresponding Galois action as $\sig(\eps,Q) = (\eps \cdot \left<\ovl{\alp}(\sig),Q\right>,Q)$. Now let $\ovl{\bet}': \Gam \lrar \Aff(\ovl{Z}_{\alp},\mu_2) \cong \mu_2 \times A[2]$ be the $1$-cochain $\ovl{\bet}'(\sig) = (1,\ovl{\bet}(\sig))$. Then
$$ \ovl{\bet}'(\sig) + \sig\ovl{\bet}'(\tau) - \ovl{\bet}'(\sig\tau) = \left(\left<\ovl{\alp}(\sig),\ovl{\bet}(\tau)\right>_{\lam},0\right) $$
and so
$$ \partial\bet = \alp \cup \bet \in H^2(k,\mu_2) .$$
It then follows that $\bet$ lifts to $H^1(k,\Aff(\ovl{Z}_{\alp},\mu_2))$ if and only if $\alp \cup \bet$ vanishes.

Now let $S$ be a set of places which contains a set of generators for the class group of $k$ and such that $\alp,\bet$ are unramified outside $S$. By~\eqref{e:split-2} we have that $k_{\theta}$ is an at most quadratic extension of $k_{\alp,\bet}$ which is classified by an element $u_{\theta} \in H^1(k_{\alp,\bet},\ZZ/2)$. Furthermore, if we replace $\theta$ by $\theta' = \theta \cdot \iota_*\vphi$ then we get $u_{\theta'} = u_{\theta} + \vphi|_{k_{\alp,\bet}}$: this follows from the obvious formula
$$ \ovl{\theta}|_{k_{\alp}}' = \ovl{\theta}|_{k_{\alp}} \cdot (\iota \circ \ovl{\vphi}|_{k_{\alp}}): \Gal(k_{\theta}/k_{\alp}) \lrar \Aff(\ovl{Z}_{\alp},\mu_2) $$
relating the homomorphisms associated to the classes $\theta|_{k_{\alp}}\in H^1(k_\alp,\Aff(\ovl{Z}_{\alp},\mu_2))$, $\theta'|_{k_{\alp}}\in H^1(k_\alp,\Aff(\ovl{Z}_{\alp},\mu_2))$ and $\vphi|_{k_\alp} \in H^1(k_\alp,\mu_2)$ by Construction~\ref{c:galois}. Now for every $v \notin S$, since $k_{\theta}$ is Galois over $k$ we have that the ramification index of $k_{\theta}/k_{\alp,\bet}$ is the same for all places $u$ of $k_{\alp,\bet}$ which lie above $v$. Let $T$ denote the set of places $v$ of $k$ such that $k_{\theta}/k_{\alp,\bet}$ is ramified at all places $u$ of $k_{\alp,\bet}$ which lie above $v$. Since $S$ contains a set of generators for the class group we can find an $a \in k^*$ such that for every $v \notin S$ we have that $\val_v(a)$ is odd if and only if $v \in T$. If we now set $\theta' = \theta \cdot \iota_*([a])$ then we get that $k_{\theta}$ is unramified outside $S$, as desired.

Let us now prove (2). Let $C \in \Br(W_\alp)$ be a Brauer element whose image in $\Br(W_\alp)/\Br(k)$ is $C_\alp(\theta)$, and let $r_{\theta} = \res_{D_\alp}(C) \in H^1(D_\alp,\QQ/\ZZ)$. Since $C_\alp(\theta)$ is a $2$-torsion element it follows that $2C$ is a constant class and hence $r_{\theta}$ is a $2$-torsion element. We may hence (uniquely) consider $r_{\theta}$ as an element of $H^1(D_\alp,\ZZ/2)$. 
Let $r_{\theta}' = (r_{\theta})|_{D_\alp \otimes_k k_{\alp,\bet}} \in H^1(D_\alp \otimes_k k_{\alp,\bet},\ZZ/2)$ denote the restriction of $r_{\theta}$ and recall that we have denoted by $u_{\theta} \in H^1(k_{\alp,\bet},\ZZ/2)$ the class corresponding to the at most quadratic extension $k_{\theta}/k_{\alp,\bet}$. Since $\theta$ vanishes in $H^1(k_{\theta},\Aff(\ovl{Z}_{\alp},\mu_2))$ it follows that the image of $C$ in $\Br(W_\alp \otimes_k k_{\theta})$ is constant and hence $r_{\theta}'$ vanishes in $H^1(D_\alp \otimes_k k_{\theta},\ZZ/2)$. It then follows that $r_{\theta}'$ is either trivial or is the pullback of $u_{\theta}$. To show that $r_{\theta}'$ can only be trivial if $u_{\theta}$ is trivial we use the fact that both $r_{\theta}'$ and $u_{\theta}$ depend on the choice of $\theta$ in the same way. More precisely, if we replace $\theta$ by $\theta' = \theta \cdot \iota_*\vphi$ for some $\vphi \in H^1(k,\mu_2)$ then $\theta'$ still maps to $\bet \in H^1(k,\hat{A}[2])$, and both $r'_{\theta'} - r'_{\theta}$ and $u_{\theta'} - u_{\theta}$ are equal to the corresponding images of $\vphi$. For $u_{\theta}$ this was shown above.
As for $r'_{\theta}$, this follows from the fact that the Brauer element $C_{\alp}(\iota_*\vphi)$ can be identified with the image of the cup product $\vphi \cup [p_\alp] \in H^2(W_\alp,\mu_2)$, and hence the residue of $C_{\alp}(\iota_*\vphi)$ along $D_\alp$ is the image of $\vphi$. It will now suffice to show that $r_{\theta}' = u_{\theta}$ for just a single $\theta$ which lifts $\bet$. As such we may choose $\theta$ so that both $u_{\theta}$ and $r_{\theta}'$ are non-zero, in which case they must coincide by the above considerations (i.e., since $r_{\theta}'$ vanishes in $H^1(D_\alp \otimes_k k_{\theta},\ZZ/2)$). 

\end{proof}

We finish this section with the following lemma, which gives some information on the way the Brauer elements $C_\alp(\theta)$ pair with local points in certain circumstances. For clarity of exposition we remark that the main purpose of Lemma~\ref{l:const} is to handle places which are part of a fixed $2$-structure for $A$, although this information is not needed for the formulation or proof of this lemma.
\begin{lem}\label{l:const}
Let $v$ be a finite odd place of $k$, let $A$ be an abelian variety over $k_v$ such that the Galois module $A[2]$ is unramified, and let $\alp \in H^1(k_v,A[2])$ be an unramified element (so that, in particular, $Y_\alp(k_v) \neq \emptyset$, and hence $X_\alp(k_v) \neq \emptyset$). Let $R \subseteq W_\alp(k_v)$ be the subset consisting of those points $x \in W_\alp(k_v)$ which lift to $U^F_\alp(k_v)$ for some \textbf{unramified} quadratic extension $F/k_v$. Let $\theta \in H^1(k_v,\Aff(\ovl{Z}_{\alp},\mu_2))$ be an unramified element and $C \in \Br(W_\alp)[2]$ be a $2$-torsion Brauer element whose image in $\Br(W_\alp)/\Br(k)$ is $C_\alp(\theta)$. Then the evaluation map $\ev_C: R \lrar \ZZ/2$ restricted to $R$ is constant.
\end{lem}
\begin{proof}
Let $\cA \lrar \spec(\OO_v)$ be a Néron model for $A$ and let $\cA[2] \subseteq \cA$ be the Zariski closure of $A[2]$. Then $\cA[2]$ is smooth subscheme of $\cA$ and the blow-up $\wtl{\cA}$ of $\cA$ along $\cA[2]$ is a regular $\cO_v$-model for $\wtl{A}$. The antipodal incolution $\iota_A: A \lrar A$ extends to an involution $\iota_{\cA}: \cA \lrar \cA$ and consequently to an involution $\iota_{\wtl{\cA}}: \wtl{\cA} \lrar \wtl{\cA}$. Since $\cA$ is quasi-projective (see~\cite{BLR}) so is $\wtl{\cA}$, and hence the geometric quotient $\cX$ of $\wtl{\cA}$ by $\iota_{\wtl{\cA}}$ exists as a scheme. Furthermore, since we assume that $v$ is odd the corresponding action of $\ZZ/2$ is tame and hence $\cX$ is also a \textbf{universal} geometric quotient (see~\cite{CEPT96}). This means, in particular, that the special fiber of $\cX$ is the geometric quotient of the special fiber of $\wtl{\cA}$ by the associated involution. A local calculation then shows that $\wtl{\cA}/\iota_{\wtl{\cA}}$ is in fact a smooth scheme, yielding a regular $\OO_v$-model for the Kummer surface $X = \Kum(A)$. Similarly, since $\alp \in H^1(k_v,A[2])$ is an unramified class we may naturally consider $\alp$ as an element of $H^1(\OO_v,\cA[2])$ and consequently twist $\cA$ by $\alp$. This results in a regular $\OO_v$-model $\cY_\alp$ for $Y_\alp$, whose special fiber $(\cY_\alp)_{\FF_v}$ is a torsor under the special fiber $\cA_{\FF_v}$ of $\cA$ associated to the reduction $\ovl{\alp} \in H^1(\FF_v,\cA[2]_{\FF_v})$. The induced antipodal involution $\iota_{Y_\alp}: Y_{\alp} \lrar Y_{\alp}$ extends to an antipodal involution $\iota_{\cY_\alp}: \cY_\alp \lrar \cY_\alp$ whose fixed locus is the Zariski closure $\cZ_\alp \subseteq \cY_\alp$ of $Z_\alp \subseteq Y_\alp$. Blowing-up $\cY_\alp$ along $\cZ_\alp$ and taking the corresponding (universal geometric) quotient by the induced involution we obtain a regular $\cO_v$-model $\cX_\alp = \wtl{\cY}_{\alp}/\iota_{\wtl{Y}_{\alp}}$ for the Kummer surface $X_\alp = \Kum(Y_\alp)$. The exceptional locus $\cD_\alp \subseteq \wtl{\cY}_\alp$ is then a regular $\cO_v$-model for $D_\alp$, and by a mild abuse of notation we will identify it will its image in $\cX_\alp$. 

We note that the complement of $W_\alp$ in $\cW_\alp$ is the union of the divisor $\cD_\alp \subseteq \cW_\alp$ and the special fiber $(\cW_\alp )_{\FF_v} \subseteq \cW_\alp$. Let $x \in W_\alp(k_v)$ be a point which lifts to $U^F_\alp(k_v)$ for some unramified quadratic extension $F/k_v$, and let $\xmcal{x}$ be its Zariski closure in $\cW_\alp$. Since $p^F_\alp: U^F_\alp \lrar W_\alp$ is ramified along $D_\alp$ it follows that each intersection point of $\xmcal{x}$ and $\cD_\alp$ must have even multiplicity. Now the Brauer element $C$ has order $2$ and is (possibly) ramified only over the complement $\cW_\alp \setminus W_\alp = \cD_\alp \cup (\cW_\alp)_{\FF_v}$. Since each intersection point of $\xmcal{x}$ and $\cD_\alp$ has even multiplicity the residue of $\xmcal{x}^{*}C \in \Br(\spec(k_v))$ along $\spec(\FF_v)$ depends only on the residue of $C$ along $(\cW)_{\FF_v}$. Since $\theta \in H^1(k_v,\Aff(\ovl{Z}_{\alp},\mu_2))$ is assumed to be unramified we see that the element $C$ becomes constant in $\Br(W_\alp \otimes_{k_v} k^{\un}_v)$ and hence the residue of $C$ along $(\cW_\alp)_{\FF_v}$ vanishes in $(\cW_\alp)_{\FF_v} \otimes_{\FF_v} \ovl{\FF_v}$. It then follows that $\res_{(\cW_\alp)_{\FF_v}}(C)$ is constant, and so we may conclude that the evaluation map $\ev_C: R \lrar \ZZ/2$ is constant, as desired.
\end{proof}

\subsection{Quadratic twists and the Mazur-Rubin lemma}\label{s:MR}

Let $A$ be an abelian variety over $k$ equipped with a principal polarization $\lam: A \x{\cong}{\lrar} \hat{A}$ which is induced by a symmetric line bundle on $A$. Let $\alp \in H^1(k,A[2])$ be an element and let $Y_\alp$ be the associated $2$-covering of $A$. Then $Y_\alp$ carries an adelic point if and only the image $[Y_\alp] \in H^1(k,A)$ of $\alp$ lies in $\Sha(A)$. Recall that the \textbf{Selmer} group $\Sel_2(A) \subseteq H^1(k,A[2])$ is defined as the preimage of $\Sha(A) \subseteq H^1(k,A)$ under the natural map $H^1(k,A[2]) \lrar H^1(k,A)$. We then have a short exact sequence
$$ 0 \lrar A(k)/2A(k) \lrar \Sel_2(A) \lrar \Sha(A)[2] \lrar 0 .$$
Given a quadratic extension $F/k$ we may canonically identify $H^1(k,A[2])$ with $H^1(k,A^F[2])$, and consequently consider the Selmer groups $\Sel_2(A^F)$ for all $F/k$ as subgroup of the same group $H^1(k,A[2])$. In order to use Swinnerton-Dyer's method in the proof of the main theorem, we will need to know how the Selmer group changes when one makes sufficiently simple quadratic twists. For this purpose we will use an approach developed by Mazur and Rubin for analyzing the behavior of Selmer groups in families of quadratic twists (see~\cite[\S 3]{MR10}).

For a place $v$ of $k$ and a quadratic extension $F/k$, let $W^F_v \subseteq H^1(k_v,A[2])$ be the kernel of the map $H^1(k,A[2]) = H^1(k,A^F[2]) \lrar H^1(k,A^F)$. The Selmer group $\Sel_2(A^F) \subseteq H^1(k,A[2])$ is then determined by the condition that $\loc_v(x) \in W^F_v$ for every place $v$. When $F$ is the trivial quadratic extension we will denote $W^F_v$ simply by $W_v$. The intersection $U_v = W_v \cap W^F_v$ is then a measure of the difference between the Selmer conditions before and after a quadratic twist by $F$. It is also useful to encode this information via the corresponding quotients $\ovl{W}_v = W_v/U_v$ and $\ovl{W}_v^F = W_v^F/U_v$. Given a finite set of places $T$ we will denote by $\ovl{W}^F_T = \oplus_{v \in T} \ovl{W}^F_v$ and by $V^F_T \subseteq \ovl{W}^F_T$ the image of $\Sel_2(A^F)$. As above, when $F$ is the trivial extension we will simply drop the supscript $F$ from the notation, yielding $\ovl{W}_T$ and $V_T$.

For each place $v$ of $k$, the Weil pairing~\eqref{e:weil} induces a local alternating pairing
\begin{equation}\label{e:local} 
\cup_{v}: H^1(k_v,A[2]) \times H^1(k_v,A[2]) \lrar H^2(k_v,\mu_2) \x{\inv}{\cong} \ZZ/2 
\end{equation}
Local arithmetic duality for abelian varieties asserts that the pairing~\eqref{e:local} is non-degenerate and admits $W_v \subseteq H^1(k_v,A[2])$ as a maximal isotropic subspace. In particular, $\dim_2 W_v = \dim_2 W_v^F$. 

\begin{rem}\label{r:alternating-2}
In general $\cup_{v}$ not alternating. Instead, one has the formula
$$ x \cup_{v} x = x \cup_{v} c_\lam, $$
where $c_\lambda \in H^1(k_v,A[2])$ is the local image of the obstruction element for realizing $\lam$ as the polarisation arising from a symmetric line bundle defined over $k$ (see~\cite[Theorem 3.4]{PR}). As we assumed that $\lam$ comes from a symmetric line bundle in our case the element $c_\lam$ vanishes and $\cup_{v}$ is alternating.
\end{rem}

\begin{lem}\label{l:good-red}
Let $v$ be a place of good reduction for $A$ and let $F$ be a quadratic extension which is ramified at $v$. Then $U_v = 0$.
\end{lem}
\begin{proof}
See~\cite[Lemma 4.3]{HS15}.
\end{proof}

\begin{lem}\label{l:mult}
Let $w$ be a place which belongs to the $2$-structure of $A$ and let $F$ be a quadratic extension such that $w$ is inert (and in particular unramified) in $F$. Then $\dim_2 \ovl{W}_w = \ovl{W}^F_w = 1$. Furthermore, the intersection $W_w \cap W_w^F$ contains exactly the elements of $W_w$ (or $W_w^F$) which are unramified.
\end{lem}
\begin{proof}
Since $F$ is unramified at $w$ the components groups $C_w$ and $C^F_w$ of $A$ and $A^F$ respectively are naturally isomorphic. To compute $W_w \cap W^F_w$ we use Lemma $4.1$ of~\cite{HS15} which asserts that
$$ W_w \cap W^F_w = \del(\mN(A(F_w))) $$
where $F_w = F \otimes_k k_w$ and $\mN: A(F_w) \lrar A(k_w)$ is the norm map. Since $w$ is part of a $2$-structure the group $C_w$ is cyclic and $|C_w| = 2$ mod $4$ (see Remark~\ref{r:odd}). Combining~\cite[Proposition $4.2$, Proposition $4.3$]{Ma}, and using the fact that $A$ is isomorphic to its dual by the principal polarization $\lam$, we may deduce that 
$$ A(k_w)/\mN(A(F_w)) \cong \ZZ/2 .$$
On the other hand, since $2A(k_w) \subseteq \mN(A(k_w))$ the map $\del$ induces an isomorphism
$$ A(k_w)/\mN(A(F_w)) \cong \del(A(k_w))/\del(\mN(A(F_w))) \cong W_w/(W_w \cap W^F_w) $$
and so the latter group is isomorphic to $\ZZ/2$, as desired. Finally, let us note that since $F/k_w$ is unramified the base change $A_F$ also has a multiplicative reduction at $w$ with component group $C^F_w \cong C_w$. In particular $C^F_w/2C^F_w \cong \ZZ/2$ has trivial Galois action and so every point in $\mN(A(F_w))$ reduces to a component in $2C_w$. Since $A(k_w)/\mN(A(F_w)) \cong \ZZ/2$ it follows that this condition is sufficient as well, i.e., the points of $A(k_w)$ which are norm from $A(F_w)$ are exactly those whose image in $C_w/2C_w$ is trivial. On the other hand, by Hensel's lemma these are also exactly the points which are divisible by $2$ in $A(k^{\un}_w)$, and hence exactly the points $x \in A(k_w)$ such that $\del(x)$ is unramified.
\end{proof}

\begin{rem}\label{r:semi-direct}
Let $w$ be a place which belongs to the $2$-structure $M$ of $A$. Combining Lemma~\ref{l:mult} and Corollary~\ref{c:p_w-3} we may conclude that the Selmer condition subspace $W_w \subseteq H^1(k_w,A[2])$ is generated over $W_w \cap H^1(\cO_w,A[2])$ by the element $\partial(Q_w)$. This implies that every element of $\Sel_2(A)$ can be written uniquely as a sum of an element unramified over $M$ and an element in the image of $A[2]$.
\end{rem}

Now let $T$ be such that $W_v = W^F_v$ for every $v \notin T$. Then the kernel of the surjective map $\Sel_2(A) \lrar V_T$ can be identified with the kernel of the surjective map $\Sel_2(A^F) \lrar V^F_T$, and hence
$$ \dim_2(\Sel(A^F)) - \dim_2(\Sel(A)) = \dim_2(V_T^F) - \dim_2(V_T) .$$
The following lemma, which is based on the ideas of Mazur and Rubin for analyzing the behavior of Selmer groups in families of quadratic twists (see~\cite[\S 3]{MR10}), is our key tool for controlling the difference $\dim_2(\Sel(A^F)) - \dim_2(\Sel(A))$ after quadratic twists.
\begin{lem}[Mazur-Rubin]\label{l:mazur-formula}
Let $A$ be as above. Let $F/k$ be a quadratic extension and let $T$ be a finite set of odd places of $k$ such that $W_v = W^F_v$ for every $v \notin T$. Let $r = \dim_2\ovl{W}_T = \dim_2\ovl{W}^F_T$. Then
$$ \dim_2V_T + \dim_2V_T^F \leq r $$
and the gap $r - \dim_2V_T - \dim_2V_T^F$ is even.
\end{lem}
\begin{proof}
Let
\begin{equation}\label{e:local-2}
W_v \times W^F_v \lrar \ZZ/2
\end{equation}
be the restriction of the local Tate pairing~\ref{e:local}. Since $W_v$ and $W^F_v$ are both maximal isotropic with respect to~\ref{e:local} it follows that the left and right kernels of~\ref{e:local-2} can both be identified with $W_v \cap W^F_v$, and so~\ref{e:local-2} descends to a non-degenerate pairing
\begin{equation}\label{e:local-3}
\ovl{W}_v \times \ovl{W}^F_v \lrar \ZZ/2
\end{equation}
By summing over the places of $T$ we obtain a non-degenerate alternating form
\begin{equation}\label{e:local-4}
\ovl{W}_T \times \ovl{W}^F_T \lrar \ZZ/2
\end{equation}
between two vector spaces of dimension $r$. Finally, by quadratic reciprocity and the fact that $W_v = W^F_v$ for $v \notin T$ we get that the subspaces $V_T \subseteq \ovl{W}_T$ and $V^F_T \subseteq \ovl{W}^F_T$ are orthogonal to each other with respect to~\eqref{e:local-4} (although not necessarily maximally orthogonal) 
and so we obtain the bound
$$ \dim_2(V_T) + \dim_2(V_T^F) \leq r .$$

Let us now show that that the gap between $\dim_2(V_T) + \dim_2(V_T^F)$ and $r$ is even (cf. \cite[Theorem 2.3]{HW}). Since $\dim_2 \ovl{W}_v = 0$ for $v \notin T$ we see that for the purpose of this lemma we may always enlarge $T$. In particular, we may assume that $W_v = H^1(\cO_v,A[2])$ for $v \notin T$ and by global duality theory we may also insure that the group $H^1(\cO_T,A[2])$ embeds in $\sum_{v \in T} H^1(k_v,A[2])$ as a maximal isotropic subspace with respect to the sum of local cup products $\cup_{T} = \sum_{v \in T} \cup_{v}$. As explained in~\cite[\S 4]{PR12}, the pairing $\cup_{T}$ admits a quadratic enhancement, i.e., a quadratic function 
$$ q_T:  \sum_{v \in T} H^1(k_v,A[2]) \lrar \QQ/\ZZ $$ 
such that $q_T(x+y)-q_T(x)-q_T(y) = x \cup_{T} y$. Furthermore, $q_T$ vanishes on the isotropic subspaces $H^1(\cO_T,A[2])$, $\oplus_v W_v$ and $\oplus_v W^F_v$ (see \cite[Theorem 4.13]{PR12}), and so, in particular, $q_T$ admits maximal isotropic subspaces. While in general $q_T$ takes values in $\ZZ/4$, in our case $\cup_{T}$ is alternating (since Remark~\ref{r:alternating-2}), and so $q_T$ takes values in $\ZZ/2$. In particular, the pair $(\oplus_{v \in T} H^1(k_v,A[2]),q_T)$ is a finite dimensional \textbf{metabolic quadratic space}. We will now use the fact that in a metabolic quadratic space $Q$, the collection of maximal isotropic subspaces carries a natural \textbf{equivalence relation}, where two maximal isotorpic subspaces $U,U' \subseteq Q$ are equivalent if $\dim_2(U \cap U')$ has the same parity as $\dim_2U = \dim_2U'$. In particular, if $U,U',U''$ are three maximal isotropic subspaces then
\begin{equation}\label{e:equiv}
\dim_2(U \cap U') + \dim_2(U'\cap U'') + \dim_2(U'' \cap U) = \dim_2(U) \;\; (\mathrm{mod}\;2)
\end{equation}
see, e.g.,~\cite{KMR}. Now let $V_T'\subseteq \sum_{v \in T} W_v$ and $(V^F_T)' \subseteq \sum_{v \in T} W_T^F$ be the images of $\Sel_2(A)$ and $\Sel_2(A^F)$ respectively. Then $V_T'$ is also the preimage of $V_T$ and $(V_T^F)'$ is also the preimage of $V_T^F$, and so 
$$ \dim_2V_T - \dim_2V_T^F = \dim_2 V'_T - \dim_2 (V^F_T)' .$$ 
Applying~\eqref{e:equiv} to the maximal isotropic subspaces $H^1(\cO_T,A[2])$, $\oplus_v W_v$ and $\oplus_v W^F_v$ and using the fact that $W_v = H^1(\cO_v,A[2])$ for $v \notin T$ we may conclude that 
$$ \dim_2 V'_T + \dim_2 (V^F_T)' + \sum_{v \in T}\dim_2(W_v \cap W^F_v) $$
has the same parity as $\sum_{v \in T}\dim_2W_v$. It then follows that 
$$ \dim_2 V'_T + \dim_2 (V^F_T)' $$
has the same parity as $\sum_{v \in T}\dim_2\ovl{W}_v$, and so the desired result follows.
\end{proof}

The above lemma of Mazur and Rubin will be used to understand the change of Selmer group under quadratic twists. This step in Swinnerton-Dyer's method can be roughly described as performing ``$2$-descent in families''. As explained in \S\ref{s:intro}, our current application of this method includes a new step of ``second $2$-descent in families''. To this end we will need to know not only how the Selmer group changes in quadratic twists, but also how the Cassels-Tate pairing changes in quadratic twists. 

From this point on we \textbf{fix the assumption} that the Galois action on $A[2]$ is constant. By composing the Cassels-Tate pairing with the natural map $\Sel_2(A) \lrar \Sha(A)[2]$ we obtain an induced (generally degenerate) pairing
$$ \left<,\right>^{\CT}_A: \Sel_2(A) \times \Sel_2(A) \lrar \ZZ/2 .$$
We note that if $\alp,\bet \in \Sel_2(A)$ are elements which also belong to $\Sel_2(A^F)$ then the Cassels-Tate pairings $\left<\alp,\bet\right>^{\CT}_A$ and $\left<\alp,\bet\right>^{\CT}_{A^F}$ are generally different. The following proposition gives some information on the difference between $\left<\alp,\bet\right>^{\CT}_A$ and $\left<\alp,\bet\right>^{\CT}_{A^F}$. To phrase the result let us fix a finite set $S$ of places containing all the archimedean places, all the places above $2$, all the places of bad reduction for $A$, and big enough so that we can choose $\cO_S$-smooth models $\cX_\alp,\cD_\alp$ for $X_\alp$ and $D_\alp$ respectively. We then have an $\cO_S$-smooth model $\cW_\alp := \cX_\alp \setminus \cD_\alp$ for $W_\alp$ as well. For a $v \notin S$ we will denote by $(\cX_\alp)_v$ and $(\cD_\alp)_v$ the respective base changes from $\spec(\cO_S)$ to $\spec(\cO_v)$.

For the next proposition, recall that the Cassels-Tate pairing was defined using a certain homomorphism
$$ B_\alp: H^1(k,\hat{A}[2]) \lrar \Br(Y_\alp)/\Br(k) $$
as described in \S\ref{s:CT}. For every quadratic extension $F/k$ let
$$ B^F_\alp: H^1(k,\hat{A}[2]) \lrar \Br(Y^F_\alp)/\Br(k) $$
be the analogous map, constructed using the canonical isomorphism $\hat{A}[2] \cong \hat{A}^F[2]$. Recall that in~\S\ref{s:kummer} we considered a similar type of map 
$$ C_\alp: H^1(k,\Aff(\ovl{Z}_\alp,\mu_2)) \lrar \Br(W_\alp)/\Br(k), $$ 
see~\eqref{e:C-alp} and the discussion following it. In particular, if $\theta \in H^1(k,\Aff(\ovl{Z}_\alp,\mu_2))$ is an element such that $(h_\alp)_*(\theta) = \bet$, then $p_\alp^*C_\alp(\theta) = B_\alp(\bet)$ and in fact $(p^F_\alp)^*C_\alp(\theta) = B^F_\alp(\bet)$ for every $F/k$.


\begin{prop}\label{p:CT-twist}
Let $\alp,\bet \in \Sel_2(A)$ be two elements unramified over $S \setminus M$ and let $\theta \in H^1(k,\Aff(\ovl{Z}_{\alp},\mu_2))$ be an element such that $(h_\alp)_*(\theta) = \bet$, and such that the splitting field $k_{\theta}$ is unramified outside $S \setminus M$. Assume in addition that $C_\alp(\theta)$ can be represented by a Brauer element $C \in \Br(W_\alp)$ which extends to the $S$-integral model $\cW_\alp$. Let $a \in k^*$ be an element which is a unit over $S$ and a square over $S \setminus M$, and such that for each place $v$ with $\val_v(a)$ odd, the Frobenius element $\Frob_{v}(k_{\alp,\bet})$ is trivial. Then $\alp$ and $\bet$ belong to $\Sel_2(A^F)$ and
$$ \left<\alp,\bet\right>^{\CT}_{A^F} - \left<\alp,\bet\right>^{\CT}_{A} = \prod_{\val_v(a)=1\Mod\;2}\Frob_{v}(k_{\theta}/k_{\alp,\bet}) \in \Gal(k_{\theta}/k_{\alp,\bet}) \subseteq \ZZ/2 $$
\end{prop}
\begin{proof}
Let us first show that $\alp$ and $\bet$ belong to the Selmer group $\Sel_2(A^F)$ after quadratic twist. For a place $v \in S \setminus M$ we have that $a$ is a square at $v$ and hence the Selmer conditions of $A$ and $A^F$ are the same at $v$. For $w \in M$ the fact that $\alp,\bet$ satisfy the Selmer condition of $A$ at $w$ and are furthermore unramified at $w$ implies by Lemma~\ref{l:mult} that $\alp,\bet$ satisfy the Selmer condition of $A^F$ at $w$. Finally, for $v \notin S$, if $\val_v(a)$ is even then $A^F$ has good reduction at $v$ and so the Selmer condition of $A^F$ at $v$ is the same as that of $A$. On the other hand, if $\val_v(a)$ is odd then by assumption the Frobenius element $\Frob_{v}(k_{\alp,\bet})$ is trivial which means that $\alp,\bet$ restrict to $0$ in $H^1(k_v,A[2])$, and hence in particular satisfy the Selmer condition of $A^F$ at $v$. We may hence conclude that $\alp,\bet \in \Sel_2(A^F)$.

Now since $\alp$ belongs to both $\Sel_2(A)$ and $\Sel_2(A^F)$ we may find two adelic points $(x_v),(x^F_v) \in \prod_v W_\alp(k_v) \subseteq X_\alp(\AA_k)$ such that $(x_v)$ lifts to $\prod_v U_\alp(k_v) \subseteq Y_\alp(\AA_k)$ and $(x^F_v)$ lifts to $\prod_v U^F_\alp(k_v) \subseteq Y^F_\alp(\AA_k)$. Furthermore, we may insure the following:
\begin{enumerate}[(1)]
\item
For every place $v$ such that $a$ is a square at $v$ (e.g., every $v \in S \setminus M$) we may take $x^F_v = x_v$.
\item
For every $v$ such that $\val_v(a)$ is odd, we may require that the Zariski closure $\xmcal{x}^F_v \in (\cX_\alp)_v$
of $x^F_v$ intersects $(\cD_\alp)_v \subseteq (\cX_\alp)_v$ transversely at a single closed point of degree $1$.
\end{enumerate}

Let $S(a)$ denote the set of places $v$ such that $\val_v(a)$ is odd. Since $p_\alp^*C_\alp(\theta) = B_\alp(\bet) \in \Br(U_\alp)/\Br(k)$ and $(p_\alp^F)^*C_\alp(\theta) = B^F_\alp(\bet) \in \Br(U^F_\alp)/\Br(k)$ we have
$$ \left<\alp,\bet\right>^{\CT}_{A} = \sum_v \inv_v C(x_v) = \sum_{v \in S} \inv_vC(x_v) $$
and
$$ \left<\alp,\bet\right>^{\CT}_{A^F} = \sum_v \inv_v C(x^F_v) = \sum_{v \in S \cup S(a)} \inv_vC(x^F_v) .$$
Now for $v \in S \setminus M$ we have $x_v = x^F_v$ and so $C(x_v) = C(x^F_v)$. Furthermore, by Lemma~\ref{l:const} we have that $C$ evaluates to the same value on $x_w$ and $x^F_w$ for every $w \in M$. We may hence conclude that
$$ \left<\alp,\bet\right>^{\CT}_{A^F} - \left<\alp,\bet\right>^{\CT}_{A} = \sum_{v \in S(a)} \inv_vC(x^F_v) .$$
Now let $v \in S(a)$ be a place. Since $C$ extends to the $S$-integral model $\cW_\alp$ it has non-trivial residues only along $\cD_\alp$. Since $\xmcal{x}^F_v$ intersects $\cD_\alp$ transversely at a single closed point of degree $1$ we see that the residue of $\xmcal{x}^{*}C \in \Br(\spec(k_v))$ along $\spec(\FF_v)$ coincides with the restriction of $\res_{\cD_\alp}(C) \in H^1(\cD_\alp,\QQ/\ZZ)$ to the intersection point $\xmcal{x}^F_v \cap \cD_\alp$. Now since the images of $\alp$ and $\bet$ in $H^1(k_v,A[2])$ vanish it follows that the extension $k_{\alp,\bet}/k$ splits completely over $k_v$ for every $v \in S(a)$. To prove the theorem we may hence extend our scalars to $k_{\alp,\bet}$. Proposition~\ref{p:descent}(2) now tells us that the residue $\res_{\cD_\alp}(C) \in H^1(\cD_\alp,\QQ/\ZZ)$ becomes constant when restricted to $D_\alp \otimes_k k_{\alp,\bet}$ and its value there is given by the quadratic extension $k_{\theta}/k_{\alp,\bet}$. The restriction of $\res_{\cD_\alp}(C) \in H^1(\cD_\alp,\QQ/\ZZ)$ to the intersection point $\xmcal{x}^F_v \cap \cD_\alp$ is then trivial if and only if the Frobenius element $\Frob_v(k_{\theta})$ is trivial, and so the desired result follows.
\end{proof}

\section{Rational points on Kummer varieties}\label{s:proof}

Our goal in this section is to carry out the proof of Theorem~\ref{t:main}. We will do so in three steps, which are described in sections~\ref{s:fibration}, \S\ref{s:first} and \S\ref{s:second}, respectively. Each of these steps will be formalized as a proposition (see Propositions~\ref{p:fib},~\ref{p:red-sel} and~\ref{p:red-sel-3} respectively) and the proof of Theorem~\ref{t:main}, which appears in~\S\ref{ss:proof}, essentially consists of assembling these three propositions into one argument.

In the course of all three steps it will be convenient to know that the abelian varieties and associated $2$-coverings under consideration satisfy the following technical condition:
\begin{define}\label{d:adm}
Let $A$ be an abelian variety such that the Galois action on $A[2]$ is constant, let $M$ be a $2$-structure for $A$ and let $\alp \in H^1(k,A[2])$ be an element. We will say that $(A,\alp)$ is \textbf{admissible} if for every pair of functions $f: M \lrar \{0,1\}$ and $h: M \times M \lrar \{0,1\}$ such that
$$ \prod_{w \in M} \left<\alp,P_w\right>_{\lam}^{f(w)} \prod_{(w,u) \in M \times M} \left<\del(P_w),P_u\right>_{\lam}^{h(w,u)} = 1 \in H^1(k,\mu_2) $$
we also have
$$ \prod_{(w,u) \in M \times M}\left<P_w,P_u\right>_{\lam}^{h(w,u)} = 1 \in \mu_2 .$$
\end{define}


The following lemma will be used to assure that the condition of Definition~\ref{d:adm} can be assumed to hold whenever necessary.
\begin{lem}\label{l:adm}
Let $A, M$ and $\alp$ be as in Definition~\ref{d:adm}. Let $S$ be a finite set of places containing all the archimedean places, all the places above $2$ and all the places of bad reduction for $A$. Let $F = k(\sqrt{a})$ be a quadratic extension which is ramified in at least one place outside $S$. Then $(A^F,\alp)$ is admissible.
\end{lem}
\begin{proof}
Assume that $(A^F,\alp)$ is not admissible and let $(f,h) \in (\ZZ/2)^M \times (\ZZ/2)^{M \times M}$ be such that
$$ \prod_{w \in M} \left<\alp,P_w\right>_{\lam}^{f(w)} \prod_{(w,u) \in M \times M} \left<\del_F(P_w),P_u\right>_{\lam}^{h(w,u)} = 1 \in H^1(k,\mu_2) $$
but 
$$ \prod_{(w,u) \in M \times M}\left<P_w,P_u\right>_{\lam}^{h(w,u)} = -1 .$$
According to Lemma~\ref{l:twist} and Remark~\ref{r:invariant} we then have
$$ \prod_{w \in M} \left<\alp,P_w\right>_{\lam}^{f(w)} \prod_{(w,u) \in M \times M} \left<\del(P_w),P_u\right>_{\lam}^{h(w,u)} = [a] \in H^1(k,\mu_2) $$
Since $k(\sqrt{a})$ is ramified outside $S$ and $A$ has good reduction outside $S$ we obtain a contradiction. It follows that $(A^F,\alp)$ is admissible.
\end{proof}

\subsection{Quadratic twists with points everywhere locally}\label{s:fibration}

Let $A$ be an abelian variety over $k$ such that the Galois action on $A[2]$ is constant and let $\alp \in H^1(k,A[2])$ be an element. 
Let $X_\alp = \Kum(Y_\alp)$ be the Kummer variety associated to $Y_\alp$ and suppose that $X(\AA_k)^{\Br} \neq \emptyset$. In this section we will consider the problem of finding a quadratic extension $F/k$ such that $Y^F_\alp(\AA_k) \neq \emptyset$, i.e., such that $\alp \in \Sel_2(A^F)$. Furthermore, to set some prerequisite conditions for the following steps we will wish to guarantee that $Y_\alp^F$ contains an adelic point which is furthermore orthogonal to certain Brauer elements. Recall that for every quadratic extension $F/k$ we had a homomorphism
$$ B^F_\alp: H^1(k,\hat{A}[2]) \lrar \Br(Y^F_\alp)/\Br(k) $$
which can be used to define the Cassels-Tate pairing on $A^F$ of $\alp$ against any other element. Recall also that in~\S\ref{s:kummer} we considered a similar type of map 
$$ C_\alp: H^1(k,\Aff(\ovl{Z}_\alp,\mu_2)) \lrar \Br(W_\alp), $$ 
see~\eqref{e:C-alp} and the discussion following it.

\begin{define}\label{d:C-alpha}
We will denote by $\C(W_\alp) \subseteq \Br(W_\alp)/\Br(k)$ the image of $C_\alp$. Similarly, we will denote by $\C(X_\alp) \subseteq \Br(X_\alp)/\Br(k)$ the subgroup consisting of those elements whose image in $\Br(W_\alp)/\Br(k)$ lies in $\C(W_\alp)$.
\end{define}

\begin{prop}\label{p:fibration}
Let $\B \subseteq H^1(k,A[2])$ be a finite subgroup which is orthogonal to $\alp$ with respect to $\cup_{\lambda}$. If $X_\alp$ contains an adelic point which is orthogonal to $\C(X_\alp) \subseteq \Br(X_\alp)/\Br(k)$ then there exists a quadratic extension $F/k$ such that $(A^F,\alp)$ is admissible and $Y^F$ contains an adelic point which is orthogonal to $B^F_\alp(\B) \subseteq \Br(Y^F_\alp)/\Br(k)$.
Furthermore, if $M$ is a $2$-structure for $A$ such that $\alp$ is unramified over $M$ but the image of $\alp$ in $H^1(k_w,A[2])$ is non-zero for every $w \in M$ then we may choose $F$ to be unramified over $M$.
\end{prop}

The proof of Proposition~\ref{p:fibration} will require the following lemma (which is used only the guarantee the last part concerning $M$):
\begin{lem}\label{l:unramified}
Let $M$ be a $2$-structure for $A$ and let $w \in M$ be a place such that the image of $\alp$ in $H^1(k_w,A[2])$ is unramified and non-zero. If $x \in X(k_w)$ is a local point then there exists an unramified extension $F/k_w$ such that $x$ lifts to $Y^F(k_w)$ (where $Y^F$ denotes the quadratic twist of the base change of $Y$ to $k_w$).
\end{lem}
\begin{proof}
Let $F/k_w$ be the quadratic extension splitting the fiber $\wtl{Y}_x$ of the $2$-covering $\wtl{Y} \lrar X$ over the point $x$. We need to show that $F/k_w$ is unramified. Assume by way of contradiction that $F/k_w$ is ramified. Since the image of $\alp$ in $H^1(k_w,A[2])$ is unramified there exists an unramified finite extension $K/k_w$ and an isomorphism $\vphi_K: Y_{K} \cong A_{K}$ such that $\vphi_K \circ \iota_Y = \iota_A \circ \vphi$.  
Let $L$ be the compositum of $F$ and $K$. Our assumption that $F/k_w$ is ramified means that $L$ is quadratic ramified extension of $K$. Let $\sig \in \Gal(L/K)$ be the non-trivial element. 

Let $\cA_K$ be a Néron model for $A_K$ and let $\cA_L$ be a Néron model for $A_L$. Let $C_K$ and $C_L$ denote the groups of components of the special fiber of $\cA_K$ and $\cA_L$ respectively. By construction there exists a point $y \in \wtl{Y}(F)$ which maps to $x$ and such that $\sig(y) = \iota_Y(y)$. Let $y' \in Y(F)$ be the image of $y$ and let $y'' \in A(L)$ be the image of $y'$ under the induced isomorphism $\vphi_L:Y_L \x{\cong}{\lrar} A_L$. In particular, we have $\sig(y'') = \iota_A(y'') = -y''$. Since $K/k$ is unramified we have an isomorphism $C_K/2C_K \cong C_w/2C_w \cong \ZZ/2$, and by Corollary~\ref{c:p_w-2} the group $C_L/4C_L$ is cyclic of order $4$ and the induced action of $\Gal(L/K)$ on $C_L/4C_L$ is trivial. It then follows that $y''$ must reduce to a component of $C_L/4C_L$ of order $2$, and hence to a component in the image of the open inclusion $C_K \hrar C_L$ induced by base change. This, in turn, implies that $y''$ and $\sig(y'')$ have the same reduction in the special fiber of $\cA_L$, and so this reduction must be a $2$-torsion point. Since $\alp$ is unramified we may find a regular model $\cY$ for $Y_{k_w}$ whose special fiber is a $2$-covering $\cY_{\FF_w}$ for $\cA_{\FF_w}$ classified by the image $\ovl{\alp} \in H^1(\FF_w,A[2])$ of $\alp$. It then follows that the reduction of $y'$ mod $w$ determines an $\FF_w$-point of the fixed point subscheme $Z_{\ovl{\alp}} \subseteq \cY_{\FF_w}$ under the induced involution. But this is now a contradiction to our assumption that $\ovl{\alp}$ is non-zero, and so we may conclude that $F/k_w$ must be unramified.  
\end{proof}

\begin{proof}[Proof of Proposition~\ref{p:fibration}]
Let $\sY= (\wtl{Y} \times \GG_m)_{/\mu_2}$ where $\mu_2$ acts on $\wtl{Y}$ by $\iota_Y$ and on $\GG_m$ by multiplication by $-1$. Projection on the second factor induces a map $\sY \lrar \GG_m/\mu_2 \cong \GG_m$ and for $t \in \GG_m(k) = k^*$ we may naturally identify the fiber $\sY_t$ with the quadratic twist $\wtl{Y}^{k(\sqrt{t})}$ (which is birational to $Y^F$). As in~\cite[\S 5]{SSD} one can show that $\sY \lrar \GG_m$ can be compactified into a fibration $\sX \lrar \PP^1$ whose fibers over $0,\infty \in \PP^1$ are geometrically split (in the sense that they contain an irreducible component of multiplicity $1$). Furthermore, arguing again as in~\cite[\S 5]{SSD} we see that the map $p:\sX \lrar \wtl{Y}/\iota_Y = X$ induced by the projection on the first factor is birational over $X$ to the projection $X \times_k \PP^1 \lrar X$. It then follows that pullback map $p^*:\Br(X) \lrar \Br(\sX)$ is an isomorphism. 

By Proposition~\ref{p:descent} we may find a finite subgroup $\C \subseteq H^1(k,\Aff(\ovl{Z}_{\alp},\mu_2))$ such that $(h_\alp)_*(\C) = \B$. Since $X(\AA_k)^{\C(X_\alp)} \neq \emptyset$, Harari's ``formal lemma'' implies the existence of an adelic point $(x_v) \in X(\AA_k)^{\Br}$ which lies in $W$ and is orthogonal to $\C(X_\alp)$ and to $C_\alp(\theta)$ for $\theta \in \C$. For every $v \in S$ let us fix a quadratic extension $F_v = k_v(\sqrt{t_v})/k_v$ such that $x_v$ lifts to a local point $y_v \in \wtl{Y}^{F_v}(k_v)$. By virtue of Lemma~\ref{l:unramified} we may assume that $F_w/k_w$ is unramified for every $w \in M$. The collection $(t_v,y_v)$ now determines an adelic point $(x_v') \in \sX(\AA_k)$ which maps to $(x_v) \in X(\AA_k)$. Since the pullback map $p^*:\Br(X) \lrar \Br(\sX)$ is an isomorphism we may deduce that $(x_v')$ belongs to the Brauer set of $\sX$, and is furthermore orthogonal to the pullbacs of the classes $C_\alp(\theta)$ for $\theta \in \C$. By~\cite[Theorem 9.17]{HW15} there exists a $t \in k^* \subseteq \PP^1(k)$ and an adelic point $(x_v') \in \sY_t(\AA_k) = \wtl{Y}^{k(\sqrt{t})/k}(\AA_k)$ with the following properties:
\begin{enumerate}[(1)]
\item
$t$ is arbitrarily close to $t_v$ for every $v \in S$.
\item
$x_v'$ is arbitrarily close to $x_v$ for every $v \in S$.
\item
$(x_v')$ is orthogonal to $p^*C_\alp(\theta)|_{\wtl{Y}^{k(\sqrt{t})/k}} = B^F_\alp((h_\alp)_*(\theta))$ for $\theta \in \C$. 
\end{enumerate}
The quadratic extension $F = k(\sqrt{t})$ now has all the required properties.
\end{proof}

Let us now specialize to the case where $A$ carries a principal polarization $\lam: A \x{\cong}{\lrar} \hat{A}$. We will furthermore \textbf{fix the assumption} that $A$ admits a $2$-structure $M$ such that $\alp$ is unramified over $M$ but has a non-trivial image in $H^1(k_w,A[2])$ for $w \in M$. 
\begin{rem}\label{r:sym}
Then the obstruction $c_{\lam} \in H^1(k,A[2])$ to realizing $\lam$ as induced by a symmetric line bundle on $A$ (see \cite{PS99}) vanishes in light~\cite[Lemma 5.1]{HS15} and our assumption that the Galois action on $A[2]$ is constant. We may hence assume without loss of generality that $\lam$ is induced by a symmetric line bundle.
\end{rem}
 
We would like to describe a particular finite subgroup $\B \subseteq H^1(k,A[2]) \cong H^1(k,\hat{A}[2])$ to which we will want to apply Proposition~\ref{p:fibration}. Let $B_0 \subseteq A[2] \otimes A[2]$ denote the kernel of the Weil pairing map $A[2] \otimes A[2] \lrar \mu_2$. The bilinear map $(P,Q) \mapsto \left<\del(P),Q\right>_{\lam}$ (see \S\ref{s:weil}) then induces a homomorphism $T: B_0 \lrar H^1(k,\mu_2)$. We will denote by $L_{\lambda}$ the minimal field extension such that $T(\beta)$ vanishes when restricted to $L_{\lambda}$ for every $\beta \in B_0$. We will further denote by $L_{\lambda,\alp} = L_{\lambda}k_{\alp}$ the compositum of $L_{\lambda}$ with the splitting field $k_\alp$ of $\alp$. Finally, we will denote by $L_{M,\alp} \subseteq L_{\lambda,\alp}$ the maximal subextension of $L_{\lambda,\alp}$ which is unramified over $M$.

\begin{rem}
The field $L_{M,\alp}$ is invariant under replacing $A$ by a quadratic twist $A^F$.
\end{rem}

We are now ready to describe the finite subgroup $\B \subseteq H^1(k,A[2])$ we wish to apply Proposition~\ref{p:fibration} to. For this it will be convenient to employ the following terminology: given a field extension $K/k$ be a field extension we will say that an element $\bet \in H^1(k,A[2])$ is \textbf{$K$-restricted} if $\bet|_{K} = 0 \in H^1(K,A[2])$. We will denote by $\Sel^K_2(A) \subseteq \Sel_2(A)$ the subgroup consisting of $K$-restricted elements. 
\begin{define}\label{d:good-choice}
We will denote by $\B_\alp \subseteq H^1(k,A[2])$ the finite subgroup consisting of those elements $\bet \in H^1(k,A[2])$ which are both $L_{M,\alp}$-restricted and satisfy $\alp \cup_{\lam} \bet = 1 \in H^1(k,\mu_2)$.
\end{define}

The following proposition summarizes the main outcome of this section. 

\begin{prop}\label{p:fib}
If $X_\alp(\AA_k)^{\C(X_\alp)} \neq \emptyset$ then there exists a quadratic extension $F/k$ such that $(A^F,\alp)$ is admissible, $\alp$ belongs to $\Sel_2(A^F)$, and $\alp$ is orthogonal to $\Sel^{L_{M,\alp}}_2(A^F)$ with respect to the Cassels-Tate pairing. Furthermore, if $\alp$ is unramified over $M$ but the image of $\alp$ in $H^1(k_w,A[2])$ is non-zero for every $w \in M$ then we may choose $F$ to be unramified over $M$.
\end{prop}
\begin{proof}
Apply Proposition~\ref{p:fibration} with the subgroup $\B_\alp \subseteq H^1(k,A[2])$ of Definition~\ref{d:good-choice}, and use the fact that any element $\bet \in \Sel^{L_{M,\alp}}_2(A^F)$ satisfies $\alp \cup_{\lam} \bet = 1 \in H^1(k,\mu_2)$ by local duality.
\end{proof}

\begin{rem}\label{r:2-pri-alg}
The group $\C(X_\alp) \subseteq \Br(X_\alp)/\Br(k)$ belongs in fact to $\Br_1(X_\alp)/\Br(k)$, where $\Br_1(X_{\alp})$ is the kernel of the map $\Br(X_\alp) \lrar \Br(\ovl{X}_\alp)$. Furthermore, since $\C(X_\alp)$ is a finite $2$-torsion group we can find a finite group $\C' \subseteq \Br_1(X)\{2\}$ in the $2$-primary part of $\Br_1(X)$ that maps surjectively onto $\C(X_\alp)$. We hence see that Proposition~\ref{p:fib} only needs to assume the triviality of the $2$-primary algebraic Brauer-Manin obstruction.
\end{rem}

\subsection{First descent}\label{s:first}

In this section we resume all the notation of \S\ref{s:fibration}, and we keep the assumption that the Galois action on $A[2]$ is trivial, that $A$ carries a principal polarization $\lam: A \x{\cong}{\lrar} \hat{A}$ (automatically induced by a symmetric line bundle, see Remark~\ref{r:sym}), and that the Kummer surface $X_\alp = \Kum(Y_\alp)$ contains an adelic point which is orthogonal to the subgroup $\C(X_\alp) \subseteq \Br(X_\alp)/\Br(k)$. We will also, as above, assume that $A$ admits a $2$-structure $M$ such that $\alp$ is unramified over $M$ but has a non-trivial image in $H^1(k_w,A[2])$ for $w \in M$. Applying Proposition~\ref{p:fib} we may find a quadratic extension $F/k$ which is unramified over $M$ and such that
\begin{enumerate}
\item[(A1)]\label{c:A1} $(A^F,\alp)$ is admissible.
\item[(A2)]\label{c:A2} $\alp$ belongs to $\Sel_2(A^F)$ and is orthogonal to $\Sel^{L_{M,\alp}}_2(A^F)$ with respect to the Cassels-Tate pairing. \end{enumerate}
Replacing $A$ with $A^F$ and using the canonical isomorphism $\Kum(Y_\alp) \cong \Kum(Y^F_\alp)$ we may assume without loss of generality that Conditions~\hyperref[c:A1]{(A1)} and~\hyperref[c:A2]{(A2)} above already hold for $A$ and $\alp$.  Our goal in this subsection is to find a quadratic extension $F/k$ such that Conditions~\hyperref[c:A1]{(A1)} and~\hyperref[c:A2]{(A2)} still hold for $A^F$ and such that in addition $\Sel_2(A^F)$ is generated by $\Sel^{L_{M,\alp}}_2(A^F)$ and the image of $A[2]$. We will do so by showing that if this is not the case then there is always a quadratic twist making the Selmer rank decrease.

\begin{prop}\label{p:red-sel}
Let $A$ an abelian variety as above with a $2$-structure $M$ and let $\alp \in \Sel_2(A)$ be an element which is unramified over $M$ and orthogonal to $\Sel^{L_{M,\alp}}_2(A)$ with respect to the Cassels-Tate pairing. Assume that Conditions~\hyperref[c:A1]{(A1)} and~\hyperref[c:A2]{(A2)} hold for $(A,\alp)$. If $\Sel_2(A)$ is not generated by $\Sel^{L_{M,\alp}}_2(A)$ and $\del(A[2])$ then there exists a field extension $F = k(\sqrt{a})$ with $a$ is a unit over $M$ such that: 
\begin{enumerate}[(1)]
\item
Conditions~\hyperref[c:A1]{(A1)} and~\hyperref[c:A2]{(A2)} hold for $(A^F,\alp)$.
\item
$\dim_2\Sel_2(A^F) < \dim_2\Sel_2(A)$.
\end{enumerate}
\end{prop}

\begin{proof}

Let $S$ be a finite set of places containing all the archimedean places, all the places above $2$ and all the places of bad reduction for $A$ or $X_\alp$, as well as a set of generators for the class group of $k$. 
Since the Selmer condition subgroups $W_v \subseteq H^1(k_v,A[2])$ are isotropic with respect to $\cup_{v}$ it follows that for every $\bet \in \Sel_2(A)$ we have $\alp \cup_{\lam} \bet = 1 \in H^2(k,\mu_2)$. By Proposition~\ref{p:descent} we may choose a subgroup $\C_\alp \subseteq H^1(k,\Aff(\ovl{Z}_{\alp},\mu_2))$ such that $(h_\alp)_*$ maps $\C_\alp$ isomorphically onto $\Sel^{L_{M,\alp}}_2(A)$, and such that for every $\theta \in \C_\alp$ the splitting field $k_{\theta}$ is unramified outside $S \setminus M$. Furthermore, by possibly enlarging $S$ we may assume that we have an $\cO_S$-smooth $S$-integral model $\cW_\alp$ for $W_\alp$ and such that for every $\theta \in \C_\alp$ the element $C_\alp(\theta) \in \Br(W_\alp)/\Br(k)$ can be represented by a Brauer element on $W_\alp$ which extends to $\cW_\alp$. 

Our method for constructing the desired element $a \in k^*$ consists of two parts. In the first part we find two places $v_0,v_1 \notin S$ whose associated Frobenius elements in $\Gam_k$ satisfy suitable constraints. These constraints imply in particular that there exists an element $a \in k^*$ such that $\div(a) = v_0 + v_1$. In the second part of the proof we show that the quadratic extension $F = k(\sqrt{a})$ has the desired properties.

By Remark~\ref{r:semi-direct} every element of $\Sel_2(A)$ can be written uniquely as a sum of an element unramified over $M$ and an element in the image of $A[2]$. In particular, the Selmer group $\Sel_2(A)$ is generated by $\Sel^{L_{M,\alp}}_2(A)$ and $\partial(A[2])$ if and only if $\Sel^{L_{M,\alp}}_2(A)$ contains all elements which are unramified over $M$. Let us hence assume that there exists a $\bet \in \Sel_2(A)$ which is unramified over $M$ and does not belong to $\Sel^{L_{M,\alp}}_2(A)$. 

Let $V = A[2] \oplus A[2] \oplus (A[2] \otimes A[2])$ and consider the homomorphism
$$ \Phi:V \lrar H^1(k,\mu_2) $$
given by the formula $\Phi(P_0,P_1,\sum_i P_i \otimes Q_i) = \left<\alp,P_0\right>_{\lam}\cdot\left<\bet,P_0\right>_{\lam} \cdot \prod_i \left<\del(P_i),Q_i\right>_{\lam}$. Let $R \subseteq V$ be the kernel of $\Phi$ and let $k_\phi/k$ be the minimal Galois extension such that all the elements in the image of $\Phi$ vanish when restricted to $k_\phi$. Then $k_\phi/k$ is a $2$-elementary extension and we have a natural isomorphism $\Gal(k_\phi/k) \cong \Hom(V/R,\mu_2)$. 
 
Let $B_0 \subseteq A[2] \otimes A[2]$ be the kernel of the Weil pairing $A[2] \otimes A[2] \lrar \mu_2$ and let $b \in A[2] \otimes A[2]$ be an element which is not in $B_0$, so that $A[2] \otimes A[2]$ is generated over $B_0$ by $b$. Let $V_\alp \subseteq V$ be the image of the left most $A[2]$ factor. The admissibility of $(A,\alp)$ is then equivalent to the following inclusion of subgroups of $V$:
$$ R \cap (V_\alp + A[2] \otimes A[2]) \subseteq V_\alp + B_0, $$ 
which in turn is equivalent to the statement
$$ b \notin R + V_\alp + B_0 .$$
On the other hand, the fact that $\bet$ is not $L_{M,\alp}$-restricted means that there exists a $w_{\bet} \in M$ such that $\Phi(P_{w_{\bet}})$ does not belong to the subgroup of $H^1(k,\mu_2)$ spanned by $\Phi(V_\alp)$ and $\Phi(B_0)$, a statement that is equivalent to 
$$ P_{w_{\bet}} \notin R + V_\alp + B_0 .$$ 
We may hence conclude that there exists a homomorphism $h: V \lrar \mu_2$ which vanishes on $R + V_\alp + B_0$ but does not vanish on $P_{w_{\bet}}$ and does not vanish on $b$. Now since $h$ vanishes on $R$ it determines a well homomorphism $h': V/R \lrar \mu_2$ which we may consider as an element of $\Gal(k_\phi/k)$. By Chabotarev's theorem we may choose a place $v_0 \notin S$ such that $\Frob_{v_0}(k_\phi) = h'$. By construction we now have that $\left<\alp,P\right>_{\lam}$ is a square in $k_{v_0}$ for every $P \in A[2]$, that $\left<\bet,P_{w_\bet}\right>_{\lam}$ is not a square in $k_{v_0}$, and that $\left<\del(P),Q\right>_{\lam}$ is a square in $k_{v_0}$ if and only if $\left<P,Q\right>_{\lam} = 1$. We shall now proceed to choose $v_1$.

Let is fix a finite large Galois extension $L/k$ which is unramified outside $S \setminus M$ and which contains all the splitting fields $k_{\theta}$ above. Let $\fm$ be the modulus which is a product of $8$ and all the places in $S$ except $w_{\bet}$, let $k_{\fm}$ be the ray class field of $\fm$, and let us set
$L' = k_{\fm}L$. Since $S$ contains a set of generators for the class group we may find a quadratic extension $K_{w_{\bet}}/k$ which is purely ramified at $w_{\bet}$ and is unramified outside $S$. Since $L'$ is unramified at $w_{\bet}$ while $K_{w_{\bet}}$ is purely ramified at $w_{\bet}$ it follows that $K_{w_{\bet}}$ is linearly disjoint from $L'$. We may hence deduce the existence of a place $v_1 \notin S \cup \{v_0\}$ such that
\begin{enumerate}[(1)]
\item
$\Frob_{v_1}(L') = \Frob_{v_0}(L')^{-1}$.
\item
$\Frob_{v_0}(K_{w_{\bet}})\cdot\Frob_{v_1}(K_{w_{\bet}})$ is the non-trivial element of $\Gal(K_{w_{\bet}}/k) \cong \ZZ/2$.
\end{enumerate}
By property $(1)$ above we see that the divisor $v_0 + v_1$ pairs trivially with the kernel of $H^1(k,\QQ/\ZZ) \lrar H^1(k_{\fm},\QQ/\ZZ)$ and so there exists an $a \in k^*$ which is equal to $1$ mod $\fm$ and such that
$\div(a) = v_0 + v_1$.
In particular, we see that $a$ is a square at each $v \in S \bksl \{w\}$. By quadratic reciprocity we see that $a$ is \textbf{not} a square in $w$. We now claim that $F = k(\sqrt{a})$ will give the desired quadratic twist. 

Let $T = \{w_{\bet},v_0,v_1\}$. Then $W^F_v = W_v$ for every $v \notin T$. By Lemmas~\ref{l:good-red} and~\ref{l:mult} we see that $\dim_2(\ovl{W}_{w_{\bet}}) = 1$ and $\dim_2(\ovl{W}_{v_0}) = \dim_2(\ovl{W}_{v_1}) = 2g$. Using Lemma~\ref{l:mazur-formula} we may conclude that
$$ \dim_2\Sel_2\left(A^F\right) - \dim_2\Sel_2\left(A\right) = \dim_2V^F_T - \dim_2V_T $$
with
$$ \dim_2V^F_T + \dim_2V_T \leq \dim_2\ovl{W}_{v_0} + \dim_2\ovl{W}_{v_1} + \dim_2\ovl{W}_{w_{\bet}}  = 4g + 1 .$$
To show that the $2$-rank of the Selmer group decreased we hence need to show that $\dim_2V_T \geq 2g+1$.

Since $\left<\del(P_{w_0}),P_{w_1}\right>_{\lam}$ is a square in $k_{v_0}$ if and only if $\left<P_{w_0},P_{w_1}\right>_{\lam} = 1$ we deduce that the local images of $\{\del(P_w)\}_{w \in M}$ at $v_0$ are linearly independent and hence span a $2^g$-dimensional subspace of $W_{v_0} = \ovl{W}_{v_0}$, which is consequently all of $W_{v_0}$. It will hence suffice to show that the image of $\bet$ in $V_T$ is not generated by the local images of $\{\del(P_w)\}_{w \in M}$. Let $Q' \in A[2]$ be such that the local image of $\del(Q')$ and $\bet$ at $v_0$ coincides. By construction $\left<\bet,P_{w_{\bet}}\right>_{\lam}$ is not a square in $k_{v_0}$ and so $\left<\del(Q'),P_{w_\bet}\right>_{\lam}$ is not a square at $v_0$. This means that $\left<Q',P_{w_\bet}\right>_{\lam} = -1$ and so by Corollary~\ref{c:p_w-3} we know that $\left<\del(Q'),Q_{w_\bet}\right>_{\lam}$ is ramified at $w_{\bet}$. Since $\left<\bet,Q_{w_{\bet}}\right>_{\lam}$ is unramified at $w_{\bet}$ it follows that $\del(Q')$ and $\bet$ have different local images at $w_{\bet}$. We hence deduce that the image of $\bet$ in $V_T$ cannot be spanned by images of $\{\del(P_w)\}_{w \in M}$ and so $\dim(V_T) \geq 2g+1$. This implies that
$$ \dim_2\Sel_2(A^F) < \dim_2\Sel_2(A) .$$

We now claim that $\Sel^{L_{M,\alp}}_2(A^F) = \Sel^{L_{M,\alp}}_2(A)$. Let $\bet' \in H^1(k,A[2])$ be an $L_{M,\alp}$-restricted element. Then $\bet'$ is unramified over $M$ and in particular at $w_{\bet}$. By Lemma~\ref{l:mult} this means that $\bet'$ satisfies the Selmer condition of $A$ at $w_{\bet}$ if and only if $\bet'$ satisfies the Selmer condition of $A^F$ at $w_{\bet}$. By our choice of $v_0$ and $v_1$ we see that $L_{M,\alp}$ splits completely in $v_0$ and $v_1$ and so the local images of $\bet'$ are trivial at $v_0$ and $v_1$. This implies that for $L_{M,\alp}$-restricted elements the local Selmer conditions for $A$ and $A^F$ are identical at every place and so $\Sel^{L_{M,\alp}}_2(A^F) = \Sel^{L_{M,\alp}}_2(A)$.

Finally, applying Proposition~\ref{p:CT-twist} to $\alp$ and any $L_{M,\alp}$-restricted element $\bet'$, and using the fact that $\Frob_{v_1}(k_{\theta}) = \Frob_{v_0}(k_{\theta})^{-1}$ we may now conclude that $\alp$ belongs to $\Sel_2(A^F)$ and is furthermore orthogonal to $\Sel^{L_{M,\alp}}_2(A^F)$ with respect to the Cassels-Tate pairing associated to $A^F$. By Lemma~\ref{l:adm} $(A^F,\alp)$ is admissible and so so Conditions~\hyperref[c:A1]{(A1)} and~\hyperref[c:A2]{(A2)} hold for $(A^F,\alp)$, as desired.

\end{proof}

\subsection{Second descent}\label{s:second}

In this section we resume all the notation of \S\ref{s:fibration} and \S\ref{s:first}, and we keep the assumption that the Galois action on $A$ is constant, that $A$ carries a principal polarization $\lam: A \x{\cong}{\lrar} \hat{A}$, and that the Kummer surface $X_\alp = \Kum(Y_\alp)$ contains an adelic point which is orthogonal to the subgroup $\C(X_\alp) \subseteq \Br(X_\alp)/\Br(k)$. Until now we have only used the fact that $A$ possess a $2$-structure $M \subseteq \Om_k$. For the purpose of the second descent phase we will need to utilize the stronger assumption that appears in Theorem~\ref{t:main}, namely, that $A$ can be written as a product $A = \prod_i A_i$ such that each $A_i$ has an \textbf{extended} $2$-structure $M_i \subseteq \Om_k$, and such that $A_j$ has good reduction over $M_i$ for $j \neq i$. 
Applying Proposition~\ref{p:red-sel} repeatedly using $M = \cup_i M_i$ we may find a quadratic extension $F/k$, unramified over $M$, and such that 
\begin{enumerate}
\item[(B1)]\label{c:B1}
Each $(A^F,\alp)$ is admissible.
\item[(B2)]\label{c:B2}
$\alp$ belongs to $\Sel_2(A^F)$ and is orthogonal to $\Sel^{L_{M,\alp}}_2(A^F)$ with respect to the Cassels-Tate pairing.
\item[(B3)]\label{c:B3}
$\Sel_2(A^F)$ is generated by $\Sel^{L_{M,\alp}}_2(A^F)$ and $\del(A^F[2])$.
\end{enumerate}
Let $\Sel^{\circ}_2(A) \subseteq \Sel_2(A)$ denote the subgroup consisting of those elements which are orthogonal to every element in $\Sel_2(A)$ with respect to the Cassels-Tate pairing. We note that Conditions~\hyperref[c:B1]{(B1)}~\hyperref[c:B2]{(B2)} and~\hyperref[c:B3]{(B3)} imply in particular 
\begin{enumerate}
\item[(B4)]\label{c:B4}
$\alp$ belongs to $\Sel^{\circ}_2(A)$.
\end{enumerate}

Replacing $A$ with $A^F$ and using the canonical isomorphism $\Kum(Y_\alp) \cong \Kum(Y^F_\alp)$ we may assume without loss of generality that Conditions~\hyperref[c:B1]{(B1)} and~\hyperref[c:B4]{(B4)} already hold for $A$. We now observe that we have a natural direct sum decomposition $\Sel_2(A) \cong \oplus_i \Sel_2(A_i)$ 
and so we can write $\alp = \alp_1 + ... + \alp_n$ with $\alp_i \in \Sel_2(A_i) \subseteq \Sel_2(A)$. Condition~\hyperref[c:B4]{(B4)} now implies that $\alp_i \in \Sel^{\circ}_2(A)$ for every $i=1,...,n$. Our goal in this subsection is to show that under these conditions one can find a quadratic extension $F/k$ such that $\Sel^{\circ}_2(A_i^F)$ is generated by $\alp_i$ and the image of $A_i[2]$. Equivalently, we will show that 
$\Sel^{\circ}(A^F)$ is generated by $\alp_1,...,\alp_n$ and the image of $A[2]$.

We begin with the following proposition, whose goal is to produce quadratic twists which induce a prescribed change to the Cassels-Tate pairing in suitable circumstances. We note that while the Weil pairing takes values in $\mu_2$ (which we write multiplicatively), the Cassels-Tate pairing takes values in $\ZZ/2$ (which we write additively). For the purpose of the arguments in this section, it will be convenient to use the fact that the Galois modules $\mu_2$ and $\ZZ/2$ are isomorphic. Although this isomorphism is unique, since it involves a change between additive and multiplicative notation it seems appropriate to take it into account explicitly. We will hence use the notation
$$ (-1)^{(-)}: \ZZ/2 \lrar \mu_2 $$
for the isomorphism in one direction and the notation
$$ \log_{(-1)}(-): \mu_2 \lrar \ZZ/2 $$
for the isomorphism in the other direction. Given a subgroup $B \subseteq H^1(k,A[2])$ and an element $\sig \in \Gam_k$, we will denote by $\rho_{\sig}: B \lrar A[2]$ the homomorphism sending $\bet$ to $\ovl{\bet}(\sig) \in A[2]$. Here 
$$ \ovl{\bet}: \Gam_k \lrar A[2] $$ 
is the canonical homomorphism associated to $\bet$ (which can be considered as obtained via Construction~\ref{c:galois}, or simply by taking the unique $1$-cycle representing $\bet$). Given $\sig,\tau \in \Gam_k$ we will denote by $\rho_{\sig} \wedge \rho_{\tau}: B \times B \lrar \ZZ/2$ the antisymmetric form 
$$ (\rho_{\sig} \wedge \rho_{\tau})(\bet,\bet') = \log_{(-1)}\left<\rho_{\sig}(\bet),\rho_{\tau}(\bet')\right>_{\lam} + \log_{(-1)}\left<\rho_{\sig}(\bet'),\rho_{\tau}(\bet)\right>_{\lam} .$$

\begin{prop}\label{p:red-sel-2}
Let $B \subseteq \Sel_2(A)$ be a subgroup containing only elements which are unramified over $M$. 
For any two elements $\sig,\tau \in \Gam_k$ 
there exists a field extension $F = k(\sqrt{a})$, unramified over $M$, and such that
\begin{enumerate}[(1)]
\item
$\Sel_2(A^F)$ contains $B$ and $\dim_2\Sel_2(A^F) = \dim_2\Sel_2(A)$.
\item
For every $\bet,\bet' \in B$ we have $\left<\bet,\bet'\right>^{\CT}_{A^F} = \left<\bet,\bet'\right>^{\CT}_{A} + (\rho_{\sig} \wedge \rho_{\tau})(\bet,\bet')$. 
\end{enumerate}
\end{prop}
\begin{proof}

Let $S$ be a finite set of places which contains all the archimedean places, all the places above $2$, all the places of bad reduction for $A$ or $X_\alp$, as well as a set of generators for the class group. In particular, every $\bet \in B$ is unramified outside $S \setminus M$. Now for any two $\bet,\bet' \in B \subseteq \Sel_2(A)$ we have $\bet \cup_{\lam} \bet' = 1 \in H^2(k,\mu_2)$, and so by Proposition~\ref{p:descent} we may choose an element $\theta \in H^1(k,\Aff(Z_\bet,\mu_2))$ such that $(h_\bet)_*(\theta) = \bet'$ and such that the splitting field $k_{\theta}/k_{\bet,\bet'}$ is unramified outside $S\setminus M$. By possibly enlarging $S$ we may assume that we have an $\cO_S$-smooth $S$-integral model $\cW_\alp$ for $W_\alp$ and that for every $\theta$ as above the element $C_{\bet}(\theta) \in \Br(W_\alp)/\Br(k)$ can be represented by a Brauer element on $W_\alp$ which extends to $\cW_\alp$. We may consequently fix a finite large Galois extension $L/k$ which is unramified outside $S \setminus M$ and which contains all the splitting fields $k_{\theta}$.

For each $w,w' \in M$ let $K_{w,w'}$ be the quadratic extension classified by the element $\left<\del(Q_w),Q_{w'}\right> \in H^1(k,\mu_2)$. By Corollary~\ref{c:p_w-3} we have that $K_{w,w}$ is ramified at $w$ while $K_{w,w'}$ is unramified over $M$ for $w \neq w'$. Since $L$ is unramified over $M$ it follows that the compositum of the $K_{w,w}$'s is linearly independent from the compositum of $L$'s with the $K_{w,w'}$'s for $w \neq w'$. Let $\fm$ be the modulus which is a product of $8$ and all the places in $S$, and let $k_{\fm}$ be the ray class field of $\fm$. We note that $k_{\fm}$ contains all the $K_{w,w'}$ for $w,w' \in M$. Let $\eps = \sig\tau\sig^{-1}\tau^{-1} \in \Gam_k$ be the commutator of $\sig$ and $\tau$ and let $\eps_L \in \Gal(L/k)$ be its corresponding image.

Since the image of $\eps$ is trivial in any the Galois group of any abelian extension of $k$ it follows from Chabotarev's density theorem that there exists places $v_0,v_1 \in \Om_k$ such that 
\begin{enumerate}[(1)]
\item
$\Frob_{v_0}(L) = \eps_L$.
\item
$\Frob_{v_1}(L) = 1$.
\item
$v_0$ is inert $K_{w,w}$ for every $w \in M$ and splits in $K_{w,w'}$ for every $w \neq w'$.
\item
$\Frob_{v_1}(k_{\fm}) = \Frob_{v_0}(k_{\fm})^{-1}$.
\end{enumerate}
It follows from (4) that the divisor $v_0 + v_1$ pairs trivially with the kernel of $H^1(k,\QQ/\ZZ) \lrar H^1(k_{\fm},\QQ/\ZZ)$ and so there exists an $a \in k^*$ which reduces to $1$ mod $\fm$ and such that
$\div(a) = v_0 + v_1$. In particular, $a$ is a square at each $v \in S$. We now claim that $F = k(\sqrt{a})$ will give the desired quadratic twist. 

We begin by Claim (1) above. Since $a$ is a square at each $v \in S$ it follows the Selmer condition for $A$ and $A^F$ is the same for every $v \in S$. Since the image of $\eps$ is trivial in any the Galois group of any abelian extension of $k$ we have by construction that $v_0$ and $v_1$ split in $k_{\bet}$ for every $\bet \in B$. It then follows that for $v \in \{v_0,v_1\}$ we have $\loc_v\bet = 0 \in H^1(k_v,A[2])$ for every $\bet \in B$. In particular, every $\bet \in B'$ satisfies the Selmer condition of $A^F$ for every $v \in S \cup \{v_0,v_1\}$ and is unramified outside $S \cup \{v_0,v_1\}$, implying that $B \subseteq \Sel_2(A^F)$. 

To see that $\dim_2\Sel_2(A^F) = \dim_2\Sel_2(A)$ we use Lemma~\ref{l:mazur-formula} with $T = \{v_0,v_1\}$. Indeed, $W^F_v = W_v$ for every $v \notin T$ and by Lemma~\ref{l:good-red} we see that $\dim_2(\ovl{W}_{v_0}) = \dim_2(\ovl{W}_{v_1}) = 2g$. We then have by Lemma~\ref{l:mazur-formula} that
$$ \dim_2\Sel_2\left(A^F\right) - \dim_2\Sel_2\left(A\right) = \dim_2V^F_T - \dim_2V_T $$
with
$$ \dim_2V^F_T + \dim_2V_T \leq 4g $$
Since the image of $\del(A[2])$ in $V^F_T$ spans a $2g$-dimensional subspace it will suffice to show that the image of $\del(A[2])$ in $V_T$ is $2g$-dimensional as well.  But this is now a direct consequence of the fact that $\left<\del(Q_w),Q_{w'}\right>$ is a square at $v_0$ if and only if $w = w'$, by construction.

We now prove Claim (2). Fix $\bet,\bet' \in B$ and let $\theta \in H^1(k,\Aff(Z_\bet,\mu_2))$ be the element chosen above such that $(h_\bet)_*(\theta) = \bet'$. Let $\eps_{\theta} \in \Gal(k_{\theta}/k)$ be the image of $\eps_L \in \Gal(L/k)$. By Proposition~\ref{p:CT-twist} we have
$$ \left<\bet,\bet'\right>^{\CT}_{A^F} - \left<\bet,\bet'\right>^{\CT}_{A} = \Frob_{v_0}(k_{\theta}/k_{\alp,\bet}) \cdot \Frob_{v_1}(k_{\theta}/k_{\alp,\bet}) = \eps_{\theta} \in \Gal(k_{\theta}/k_{\alp,\bet}) \subseteq \ZZ/2 .$$
Recall from~\S\ref{s:kummer} the commutative diagram
\begin{equation}
\xymatrix@C=1.1em{
1 \ar[r] & \Gal(k_{\theta}/k_{\bet}) \ar[r]\ar@{^(->}[d]^{\ovl{\theta}|_{k_{\bet}}} & \Gal(k_{\theta}/k) \ar[r]\ar@{^(->}[d]^{\ovl{\theta}} & \Gal(k_{\bet}/k) \ar[r]\ar@{^(->}[d]^{\ovl{\bet}} & 1 \\
1 \ar[r] & \Aff(Z_{\bet},\mu_2) \ar[r] & \Aff(Z_{\bet},\mu_2) \rtimes A[2] \ar[r] & A[2] \ar[r] & 1 \\
}
\end{equation}
with exact rows and injective vertical maps (cf.~\eqref{e:split}). Let us write $\ovl{\theta}(\sig) = (x,\ovl{\bet}(\sig)) \in \Aff(Z_\bet,\mu_2) \rtimes A[2]$ and $\ovl{\theta}(\tau) = (y,\ovl{\bet}(\tau)) \in \Aff(Z_\bet,\mu_2) \rtimes A[2]$ for suitable $x,y  \in \Aff(Z_\bet,\mu_2)$. We may then compute that
$$ [\ovl{\theta}(\sig),\ovl{\theta}(\tau)] = (xy^{\ovl{\bet}(\sig)}x^{\ovl{\bet}(\tau)}y,0) $$
Now $x \in \Aff(Z_\bet,\mu_2)$ is an affine-linear map whose homogeneous part is $\ovl{\bet'}(\sig) \in A[2]$ and so $x\cdot x^{\ovl{\bet}(\tau)}$ is the constant affine-linear map with value $\left<\ovl{\bet'}(\sig),\ovl{\bet}(\tau)\right>_{\lam}$. Similarly, $y\cdot y^{\ovl{\bet}(\sig)}$ is the constant affine-linear map with value $\left<\ovl{\bet'}(\tau),\ovl{\bet}(\sig)\right>_{\lam}$. It then follows that the image $\eps_{\theta} \in \Gal(k_{\theta}/k)$ of $\eps$ is trivial if and only if $(\rho_{\sig} \wedge \rho_{\tau})(\bet,\bet') = 0$, and so the desired result follows. \end{proof}

We are now ready to prove the main result of this section, showing that if $\Sel_2^{\circ}(A)$ is not generated by $\alp_1,...,\alp_n$ and the image of $A[2]$ then there exists a field extension $F = k(\sqrt{a})$ such that $\Sel_2^{\circ}(A^F)$ is strictly smaller then $\Sel_2^{\circ}(A)$. 
\begin{prop}\label{p:red-sel-3}
Let $A_1,...,A_n$ be abelian varieties as above such that each $A_i$ is equipped with an extended $2$-structure $M_i$ over which $A_j$ has good reduction for $j \neq i$. Let $A = \prod_i A_i$ and let $\alp \in \Sel_2(A)$ be a non-degenerate element (see Definition~\ref{d:nondeg}) which is unramified over $M = \cup_i M_i$ and write $\alp = \sum_i \alp_i$ with $\alp_i \in \Sel_2(A_i)$. Assume that Conditions~\hyperref[c:B1]{(B1)} and \hyperref[c:B4]{(B4)} are satisfied. If $\Sel_2^{\circ}(A)$ is not generated by $\alp_1,...,\alp_n$ and the image of $A[2]$ then there exists a field extension $F = k(\sqrt{a})$ with $a$ is a unit over $M$ and such that 
\begin{enumerate}[(1)]
\item
Conditions~\hyperref[c:B1]{(B1)} and \hyperref[c:B4]{(B4)} hold for $(A^F,\alp)$.
\item
$\dim_2\Sel^{\circ}_2(A^F) < \dim_2\Sel^{\circ}_2(A)$.
\end{enumerate}
\end{prop}
\begin{proof}

Let $S$ be a finite set of places which contains all the archimedean places, all the places above $2$, all the places of bad reduction for $A$ or $X_\alp$, as well as a set of generators for the class group. In particular, every $\bet \in U$ is unramified outside $S \setminus M$. Now for any two $\bet,\bet' \in U$ we have $\bet \cup_{\lam} \bet' = 1 \in H^2(k,\mu_2)$, and so by Proposition~\ref{p:descent} we may choose an element $\theta \in H^1(k,\Aff(Z_\bet,\mu_2))$ such that $(h_\bet)_*(\theta) = \bet'$ and such that the splitting field $k_{\theta}/k_{\bet,\bet'}$ is unramified outside $S\setminus M$. By possibly enlarging $S$ we may assume that we have an $\cO_S$-smooth $S$-integral model $\cW_\alp$ for $W_\alp$ and that for every $\theta$ as above the element $C_{\bet}(\theta) \in \Br(W_\alp)/\Br(k)$ can be represented by a Brauer element on $W_\alp$ which extends to $\cW_\alp$. 

Our general strategy for proving Proposition~\ref{p:red-sel-3} is the following. We first find a suitable $a \in k^*$ which is a unit over $S$ and such that after a quadratic twist by $F = k(\sqrt{a})$ the dimension of the Selmer group $\Sel_2(A^F)$ \textbf{increases} by $1$, where the new element $\gam$ has certain favorable properties. We then use Proposition~\ref{p:red-sel-2} in order to find a second quadratic twist which suitably modifies the Cassels-Tate pairing between $\gam$ and the elements from $\Sel_2(A)$. This last part is done in a way that effectively decreases the number of elements in the Selmer group which are in the kernel of the Cassels-Tate pairing.

Let $U \subseteq \Sel_2(A)$ denote the subgroup consisting of those elements which are unramified over $M$. By Remark~\ref{r:semi-direct} we have that $\Sel_2(A)$ decomposes as a direct sum $\Sel_2(2) = U \oplus \del(A[2])$. Let $U^\circ = U \cap \Sel_2^{\circ}$. Since $\Sel_2^{\circ}$ contains $\del(A[2])$ we obtain a direct sum decomposition $\Sel_2^{\circ}(A) = U^{\circ} \oplus \del(A[2])$. Similarly, for every $i=1,...,n$ we have a direct sum decomposition $\Sel_2(A_i) = U_i \oplus \del(A_i[2])$ and $\Sel^{\circ}_2(A_i) = U^{\circ}_i \oplus \del(A_i[2])$, where $U_i = U \cap \Sel_2(A_i)$ and $U^{\circ}_i = U^{\circ} \cap \Sel_2(A_i)$.

Let $\bet \in U^{\circ}$ be an element which does not belong to the subgroup of $U^{\circ}$ spanned by $\alp_1,...,\alp_n$ and let us write $\bet = \sum_i \bet_i$ with $\bet_i \in \Sel_2(A_i)$. It then follows that $\bet_i \in U^{\circ}_i$. Since $\bet$ is not spanned by the $\alp_i$'s there exists an $i$ such that $\bet_i \notin 0,\alp_i$. 

Let us now write $M_i = \{w_0,...,w_{2g}\}$. By the definition of an extended $2$-structure we may choose, for every $j=0,...,2g-1$ a $2$-torsion point $Q_j \in A_i[2]$ such that the image of $Q_j$ in $C_{w_{j'}}/2C_{w_{j'}}$ is non-trivial if and only if $j = j',j'+1$. Similarly, let $Q_{2g}$ be such that the image of $Q_j$ in $C_{w_{j'}}/2C_{w_{j'}}$ is non-trivial if and only if $j = 2g,0$. We note that by construction $\sum_{j=0}^{2g} Q_j = 0$. 

We now claim that there exists a $j \in \{0,...,2g\}$ such that $\left<\bet_i,Q_j\right>_{\lam}$ is non-trivial and different from $\left<\alp_i,Q_j\right>_{\lam}$. Indeed, assume otherwise and let $J \subseteq \{0,...,2g\}$ be the subset of those indices for which $\left<\bet_i,Q_j\right>_{\lam} = \left<\alp_i,Q_j\right>_{\lam}$ (so that $\left<\bet_i,Q_j\right>_{\lam} = 0$ for $j \notin J$). Then 
$$ \prod_{j \in J} \left<\alp_i,Q_j\right>_{\lam} = \prod_{j=0}^{2g} \left<\bet_i,Q_j\right>_{\lam} = \left<\bet_i,\sum_{j=0}^{2g} Q_i\right>_{\lam} = 1 \in H^1(k,\mu_2) .$$
Since $\bet_i \neq 1,\alp_i$ we have that $\emptyset \subsetneq J \subsetneq \{0,...,2g\}$, and so we obtain a contradiction to our assumption that $\alp$ is non-degenerate. We may hence conclude that $\left<\bet_i,Q_j\right> \neq 1,\left<\alp_i,Q_j\right>$ for some $j=0,...,2g$. To fix ideas let us assume that we have this for $j=2g$. We shall now remove $w_{2g}$ from $M_i$ and work with the resulting $2$-structure $M_i' = M_i \setminus \{w_{2g}\} = \{w_0,...,w_{2g-1}\}$. Let $\{P_w\}_{w \in M'}$ and $\{Q_w\}_{w \in M_i'}$ be the corresponding dual bases of $A_i[2]$. By comparing images in $\oplus_{w \in M_i'} C_{w}/2C_{w}$ we see that the $2$-torsion point $Q_{w_0} \in A_i[2]$ coincides with the point $Q_{2g}$ we had before. In particular, we have that $\left<\bet_i,Q_{w_0}\right> \neq 0,\left<\alp_i,Q_{w_0}\right>$. 

Let us now complete $M'_i$ into a $2$-structure for $A$ by choosing, for every $i' \neq i$, a $2$-structure $M'_{i'} \subseteq M_{i'}$, and setting $M' = M'_1 \cup ... \cup M'_n$. We now note that since $A_i[2]$ is orthogonal to $A_{i'}[2]$ with respect to the Weil pairing when $i \neq i'$ it follows that $\left<\bet,Q_{w_0}\right> = \left<\bet_i,Q_{w_0}\right>$ and $\left<\alp,Q_{w_0}\right> = \left<\alp_i,Q_{w_0}\right>$. We have thus found a point $Q_{w_0} \in M'$ such that $\left<\bet,Q_{w_0}\right> \neq 1,\left<\alp,Q_{w_0}\right>$. We may now forget about the factorization of $A$ into a product of the $A_i$'s, and reconsider it as a single abelian variety.


For each $w,w' \in M'$ let $K_{w,w'}$ be the quadratic extension corresponding to the element $\left<\del(Q_w),Q_{w'}\right> \in H^1(k,\mu_2)$. By Corollary~\ref{c:p_w-3} we have that $K_{w,w}$ is ramified at $w$ while $K_{w,w'}$ is unramified over $M'$ when $w \neq w'$. Let $k_U$ be the compositum of $k_{\bet'}$ for $\bet' \in U$. By our assumptions $k_U$ is unramified over $M'$ and so the compositum of the $K_{w,w}$'s is linearly independent from the compositum of $k_U$ with the $K_{w,w'}$'s for $w \neq w'$. 

Let us fix a finite large Galois extension $L/k$ which is unramified outside $S \setminus M$ and which contains all the splitting fields $k_{\theta}$. Let $\fm$ be the modulus which is a product of $8$ and all the places in $S \setminus \{w_0\}$, and let $k_{\fm}$ be the ray class field of $\fm$. We note that $k_{\fm}$ contains all the $K_{w,w'}$ for $(w,w') \neq (w_0,w_0)$ as well as $k_U$. Let $L' = k_{\fm}L$. Since $L'$ is unramified over $M'$ we see that $L'$ is linearly independent from the compositum of all the $K_{w,w}$. By Chabotarev's density theorem that there exists places $v_0,v_1 \in \Om_k$ such that 
\begin{enumerate}[(1)]
\item
$v_0$ splits in $k_U$.
\item
$v_0$ is inert $K_{w,w}$ for every $w \in M'$ and splits in $K_{w,w'}$ for every $w \neq w'$.
\item
$v_1$ splits in $K_{w_0,w_0}$.
\item
$\Frob_{v_1}(L') = \Frob_{v_0}(L')^{-1}$.
\end{enumerate}
It then follows that the divisor $v_0 + v_1$ pairs trivially with the kernel of $H^1(k,\QQ/\ZZ) \lrar H^1(k_{\fm},\QQ/\ZZ)$ and so there exists an $a \in k^*$ which reduces to $1$ mod $\fm$ and such that
$\div(a) = v_0 + v_1$. In particular, $a$ is a square at each $v \in S \setminus \{w_0\}$. By quadratic reciprocity we have that $a$ is not a square at $w_0$. This means, in particular, that $v_1$ splits in $K_{w,w'}$ if and only if $w \neq w'$ or $w=w'=w_0$. Since $k_U \subseteq k_{\fm}$ Properties (1) and (4) together imply that $k_U$ splits completely at $v_1$. Applying Proposition~\ref{p:CT-twist} we now get that $U \subseteq \Sel_2(A^F)$ and that the Cassels-Tate pairing between every two elements $\bet,\bet' \in U$ is the same in $A^F$ and $A$.

Let $T = \{w_0,v_0,v_1\}$. We then have that $W^F_v = W_v$ for every $v \notin T$ and by Lemmas~\ref{l:good-red} and~\ref{l:mult}  we see that $\dim_2(\ovl{W}_{v_0}) = \dim_2(\ovl{W}_{v_1}) = 2g$ and $\dim_2(\ovl{W}_{w_0}) = 1$. By Lemma~\ref{l:mazur-formula} we have that
$$ \dim_2\Sel_2\left(A^F\right) - \dim_2\Sel_2\left(A\right) = \dim_2V^F_T - \dim_2V_T $$
with
$$ \dim_2V^F_T + \dim_2V_T \leq \dim_2(\ovl{W}_{v_0}) + \dim_2(\ovl{W}_{v_1}) +  \dim_2(\ovl{W}_{w_0}) = 4g + 1 $$
and $4g+1 - \dim_2V^F_T - \dim_2V_T \geq 0$ is even. Now the image of $\del(A[2])$ in $V^F_T$ spans a $2g$-dimensional subspace and since $\left<\del(Q_w),Q_{w'}\right>$ is a square at $v_0$ if and only if $w = w'$ we have that the image of $\del(A[2])$ spans a $2g$-dimensional subspace of $V_T$ as well. On the other hand for every $\bet' \in U$ we have $\loc_{v_0}(\bet') = \loc_{v_1}(\bet') = 0$ and $\loc_{w_0}(\bet') \in W_{w_0} \cap W^F_{w_0}$ and so the image of $V_T$ is exactly $2g$. The parity constraint of Lemma~\ref{l:mazur-formula} now forces $\dim_2V^F_T$ to be $2g+1$. Since we saw that $\Sel_2(A^F)$ contains $U$ it now follows that $\Sel_2(A^F)$ is generated by $U$, $\del^F(A[2])$ and one more element $\gam \in \Sel_2(A^F)$. Furthermore, by adding to $\gam$ an element of $\del(A^F[2])$ we may assume that $\gam$ is unramified over $M'$. 
\begin{lem}
There exists an element $\sig \in \Gam_k$ such that $\ovl{\gam}(\sig) = Q_{w_0}$ and such that $\ovl{\bet'}(\sig) = 0$ for every $\bet' \in U$. 
\end{lem}
\begin{proof}
First observe that since $\gam$ is unramified over $M$ we have that $\loc_{w_0}(\gam) \in W_{w_0} \cap W^F_{w_0}$ and since the image of $\gam$ in $V^F_T$ is orthogonal to $V_T$ with respect to~\eqref{e:local-4} we may conclude, in particular, that
$$ \inv_{v_0}\left[\gam \cup_{\lam} \del(P)\right] = \inv_{v_1}\left[\gam \cup_{\lam} \del(P)\right] $$
for every $P \in A[2]$. Using the mutually dual bases $Q_w$ and $P_w$ we may write this equality as
$$ \inv_{v_0}\sum_{w' \in M'}\left[\left<\gam,P_{w'}\right>_{\lambda}\cup \left<\del(P),Q_{w'}\right>_{\lambda}\right] = \inv_{v_1}\sum_{w' \in M'}\left[\left<\gam,P_{w'}\right>_{\lambda}\cup \left<\del(P),Q_{w'}\right>_{\lambda}\right] .$$
Let us now plug in $P = Q_w$ for some $w \in M$. By construction we have that $\left<\del(Q_w),Q_{w'}\right>_{\lambda}$ vanishes at both $v_0$ and $v_1$ whenever $w \neq w'$ and so we obtain
$$ \inv_{v_0}\left[\left<\gam,P_{w}\right>_{\lambda}\cup \left<\del(Q_w),Q_{w}\right>_{\lambda}\right] = \inv_{v_1}\left[\left<\gam,P_{w}\right>_{\lambda}\cup \left<\del(Q_w),Q_{w}\right>_{\lambda}\right] .$$
Now if $w \neq w_0$ then $\left<\del(Q_w),Q_{w}\right>_{\lambda}$ is unramified and non-trivial at both $v_0$ and $v_1$ and so for such $w$ we obtain
\begin{equation}\label{e:val}
\val_{v_0}\left<\gam,P_{w}\right>_{\lambda} = \val_{v_1}\left<\gam,P_{w}\right>_{\lambda} \;\;\Mod\; 2
\end{equation}
while if $w = w_0$ then $\left<\del(Q_w),Q_{w}\right>_{\lambda}$ is unramified and non-trivial at $v_0$ but is trivial at $v_1$, and so we obtain
$$ \val_{v_0}\left<\gam,P_{w_0}\right>_{\lambda} = 0 \;\;\Mod\; 2 $$
%
We now observe that $\val_{v_1}\left<\gam,P_{w_0}\right>_{\lam}$ must be odd. Indeed, otherwise Equation~\eqref{e:val} would hold for all $w \in M'$, and so there would exist a $Q \in A[2]$ such that $\gam' = \gam + \del^F(Q)$ is unramified, and hence trivial, at both $v_0$ and $v_1$.
It would then follow that $\gam'$ satisfies the Selmer condition of $\Sel_2(A)$ at all places except possibly $w_0$. Since the image of $\gam'$ in $V^F_T$ is orthogonal to the image of $\del(Q_{w_0})$ in $V_T$ to~\eqref{e:local-4} we may conclude that $\gam'$ satisfies the Selmer condition of $A$ at $w_0$ as well, i.e., $\gam' \in U \subseteq \Sel_2(A) \cap \Sel_2(A^F)$. But this would imply that $\Sel_2(A^F)$ is generated by $U$ and $\del^F(A[2])$, contradicting the above. We may hence conclude that $\val_{v_1} \left<\gam,P_{w_0}\right>_{\lam}$ is odd and so
$$ \val_{v_0} \left<\gam,P_{w}\right>_{\lam} + \val_{v_1}\left<\gam,P_{w}\right>_{\lam} = \left\{\begin{matrix} 0 \in \ZZ/2 & w \neq w_0 \\ 1 \in \ZZ/2 & w = w_0 \\ 
\end{matrix}\right. .$$
It follows that $\left<\gam,P_{w_0}\right>_{\lam}$ does not belong to the minimal field extension of $k_U$ splitting the classes $\{\left<\gam,P_{w}\right>_{\lam}\}_{w \neq w_0}$. Consequently, there must exist an element $\sig \in \Gam_k$ such that $\sig \in \Gam_{k_U} \subseteq \Gam_k$, such that
$\left<\ovl{\gam}(\sig),P_w\right>_{\lambda}$ is trivial for $w \neq w_0$ while $\left<\ovl{\gam}(\sig),P_{w_0}\right>_{\lambda}$ is non-trivial. In particular, $\ovl{\gam}(\sig) = Q_{w_0}$ and $\ovl{\bet'}(\sig) = 0$ for every $\bet' \in U$. 
\end{proof}

Now recall that we have an element $\bet \in U$ such that $\left<\bet,Q_{w_0}\right>_{\lam}$ and $\left<\alp,Q_{w_0}\right>_{\lam}$ are \textbf{two different non-trivial} classes in $H^1(k,\mu_2)$. It follows that for any two elements $\eps_{\alp},\eps_{\bet} \in \mu_2$ there exists an element $\tau \in \Gam_k$ such that $\left<\ovl{\alp}(\tau),Q_{w_0}\right>_{\lam} = \eps_{\alp}$ and $\left<\ovl{\bet}(\tau),Q_{w_0}\right>_{\lam} = \eps_{\bet}$. For our purposes let us set $\eps_{\alp} = (-1)^{\left<\alp,\gam\right>^{\CT}_{A^F}}$ and $\eps_{\bet} = (-1)^{1-\left<\bet,\gam\right>^{\CT}_{A^F}}$. Let $B \subseteq \Sel_2(A^F)$ be the subgroup generated by $U$ and $\gam$ and let $\rho_{\sig} \wedge \rho_{\tau}: B \times B \lrar \ZZ/2$ be the alternating form constructed above. Since $\rho_{\sig}(\bet') = \ovl{\bet'}(\sig) = 0$ for every $\bet' \in U$ it follows that $\rho_{\sig} \wedge \rho_{\tau}(\bet',\bet'') = 0$ for every $\bet',\bet'' \in U \subseteq B$. On the other hand, since $\rho_{\tau}(\gam) = \ovl{\gam}(\tau) = Q_{w_0}$ we have $\rho_{\sig} \wedge \rho_{\tau}(\bet',\gam) = \log_{(-1)}\left<\ovl{\bet'}(\tau),Q_{w_0}\right>_{\lam}$ for every $\bet' \in U$.
Applying Proposition~\ref{p:red-sel-2} with $B$ the subgroup generated by $U$ and $\gam$ and with the elements $\sig,\tau \in \Gam_k$ constructed above we obtain a quadratic twist $F' = k(\sqrt{a'})$ such that (with $a'' := a'a$ and $F'' := k(\sqrt{a''})$) we have
\begin{enumerate}[(1)]
\item
$\Sel_2(A^{F''})$ contains $U$ and $\gam$ and $\dim_2\Sel_2(A^{F''}) = \dim_2\Sel_2(A^{F})$.
\item
For every $\bet',\bet'' \in U$ we have $\left<\bet',\bet''\right>^{\CT}_{A^{F''}} = \left<\bet',\bet''\right>^{\CT}_{A^F}$.
\item
For every $\bet' \in U$ we have $\left<\bet',\gam\right>^{\CT}_{A^{F''}} = \left<\bet',\gam\right>^{\CT}_{A^{F}} + (\rho_{\sig} \wedge \rho_{\tau})(\bet,\bet')$. In particular $\left<\alp,\gam\right>^{\CT}_{A^{F''}} = 0$ and $\left<\bet,\gam\right>^{\CT}_{A^{F''}} \neq 0$. 
\end{enumerate}
Let $V \subseteq \Sel_2(A^{F''})$ be the subgroup consisting of those elements which are unramified over $M'$, so that we have a direct sum decomposition $\Sel_2(A^{F''}) = V \oplus \del^{F''}(A[2])$. Property (1) above implies that $V$ is generated by $U$ and $\gam \notin U$. Let $V^{\circ} = V \cap \Sel^{\circ}_2(A^{F''})$. Since all the elements of $V^{\circ}$ are in particular orthogonal to $\bet$ with respect to the Cassels-Tate pairing, Properties (2) and (3) above imply that $V^{\circ} \subseteq U$. Since all the elements of $V^{\circ}$ are also orthogonal to $\gam$ with respect to the Cassels-Tate pairing, Properties (2) and (3) further imply that $\bet \notin V^{\circ} \subseteq U$ while $\alp \in V^{\circ}$. This means in particular that Condition~\hyperref[c:B4]{(B4)} holds for $(A^{F''},\alp)$. By Lemma~\ref{l:adm} we have that $(A^{F''},\alp)$ is admissible, i.e., Condition~\hyperref[c:B1]{(B1)} holds as well. Finally, since $\Sel^{\circ}_2(A^{F''})$ is a direct sum of $V^{\circ}$ and the image of the $2$-torsion we may now conclude that $\dim_2\Sel^{\circ}_2(A^{F''}) < \dim_2\Sel^{\circ}(A)$. It follows that the quadratic extension $F''$ has the desired properties and so the proof is complete.

\end{proof}

\subsection{Proof of the main theorem}\label{ss:proof}

In this section we will complete the proof of Theorem~\ref{t:main}. Let $k$ be a number field and let $A_1,...,A_n$ be principally polarized simple abelian varieties over $k$, such that each $A_i$ has all its $2$-torsion defined over $k$. For each $i$, let $M_i \subseteq \Om_k$ be an extended $2$-structure for $A_i$ such that $A_j$ has good reduction over $M_j$ whenever $i \neq j$. Let $A = A_1 \times ... \times A_n$ and let $\alp \in H^1(k,A[2])$ be a non-degenerate element which is unramified over $M = \cup_i M_i$ but which has a non-trivial image in $H^1(k_w,A[2])$ for each $w \in M$. We may uniquely write $\alp = \sum_i \alp_i$ with $\alp_i \in H^1(k,A_i[2])$. Let $Y_i$ be the $2$-covering of $A_i$ determined by $\alp_i$ so that $Y = \prod_i Y_i$ is the $2$-covering of $A$ determined by $\alp$. Finally, let $X = \Kum(Y)$ be the associated Kummer surface.

\begin{proof}[Proof of Theorem~\ref{t:main}]
To prove that Conjecture~\ref{c:kummer} holds for $X$, let us assume that the $2$-primary Brauer-Manin obstruction to the Hasse principle is the only one for each $Y^F$, i.e., that $[Y^F] \in H^1(k,A)$ is not a non-trivial divisible element of $\Sha(A^F)$ for any $F$. Since $H^1(k,A) = \oplus_k H^1(k,A_i)$ and $\Sha(A) = \oplus_i \Sha(A_i)$ this is equivalent to saying that $[Y_i^F] \in H^1(k,A_i)$ is not a non-trivial divisible element of $\Sha(A_i^F)$ for any $F$.

In light of Lemma~\ref{l:adm} we may, by possibly replacing $A$ by a quadratic twist, assume that $(A,\alp)$ is admissible. Applying Proposition~\ref{p:fib} we may find a quadratic extension $F/k$, unramified over $M$, such that $(A^{F},\alp)$ satisfies Conditions~\hyperref[c:A1]{(A1)} and~\hyperref[c:A2]{(A2)} above. Replacing $A$ with $A^{F}$ we may assume that Conditions~\hyperref[c:A1]{(A1)} and~\hyperref[c:A2]{(A2)} already hold for $(A,\alp)$. By repeated applications of Proposition~\ref{p:red-sel} we may find a quadratic extension $F'/k$, unramified over $M$, such that $(A^{F'},\alp)$ satisfies Conditions~\hyperref[c:B1]{(B1)} and~\hyperref[c:B4]{(B4)} above. Replacing $A$ with $A^{F'}$ we may assume that Conditions~\hyperref[c:B1]{(B1)} and~\hyperref[c:B2]{(B2)} already hold for $(A,\alp)$. By repeated applications of Proposition~\ref{p:red-sel-3} we may find a quadratic extension $F''/k$, unramified over $M$, and such that the subgroup $\Sel^{\circ}_2(A^{F''}) \subseteq \Sel_2(A^{F''})$ consisting of those elements which are orthogonal to all of $\Sel_2(A^{F''})$ with respect to the Cassels-Tate pairing is generated by $\alp_1,...,\alp_n$ and the image of the $2$-torsion. It then follows that $\Sel^{\circ}_2(A_i^{F''})$ is generated by $\alp_i$ and the image of the $2$-torsion. Let $\Sha^{\circ}(A_i^{F''}) \subseteq \Sha(A_i^{F''})$ be the subgroup orthogonal to all of $\Sha(A_i^{F''})[2]$ with respect to the Cassels-Tate pairing. Then we may conclude that $\Sha^{\circ}(A_i^{F''})$ is generated by the image of $\alp_i$, i.e., by the class $[Y_i^{F''}]$ of $Y_i^{F''}$. Since we assumed that $[Y_i^{F''}]$ is not a non-trivial divisible element it now follows that $\Sha(A_i^{F''})\{2\}$ is finite. The Cassels-Tate pairing induces a non-degenerate self pairing of $\Sha(A_i^{F''})\{2\}$, which is alternating in our case by Remark~\ref{r:alternating}. This means, in particular, that if we write the abstract abelian group $\Sha(A_i^{F''})\{2\}$ as a direct sum $\oplus_i \ZZ/2^{n_i}$ of cyclic groups then for each $n$ it will have an even number of $\ZZ/2^n$ components. Now the multiplication by $2$ map induces an isomorphism $\Sha(A_i^{F''})[4]/\Sha(A_i^{F''})[2] \cong \Sha^{\circ}(A_i^{F''})$ and so by the above we may conclude that the $2$-rank of $\Sha^{\circ}(A_i^{F''})$ is even. Since it generated by a single element it must therefore vanish, implying that $[Y^{F''}] = 0$. This means that $Y^{F''}$ has a rational point and so $X$ has a rational points as well, as desired.
\end{proof}

\end{document}